\documentclass{article}

\usepackage[margin=25mm]{geometry}
\usepackage{amsmath,amssymb,amsthm,mathtools,bm}
\usepackage{booktabs}
\usepackage{enumitem}
\usepackage{hyperref}
\usepackage[nameinlink,noabbrev]{cleveref}

\hypersetup{
  colorlinks=true,
  linkcolor=blue,
  citecolor=blue,
  urlcolor=blue,
  pdftitle={Multivariate P- and Q-polynomial structures of generalized Johnson association schemes},
  pdfauthor={Hirotake Kurihara},
  pdfsubject={Association schemes, multivariate P- and Q-polynomial structures, eigenmatrices, orthogonality}
}

\numberwithin{equation}{section}
\newtheorem{theorem}{Theorem}[section]
\newtheorem*{maintheorem}{Theorem}
\newtheorem{proposition}[theorem]{Proposition}
\newtheorem{lemma}[theorem]{Lemma}

\theoremstyle{definition}
\newtheorem{definition}[theorem]{Definition}
\newtheorem{remark}[theorem]{Remark}

\crefname{theorem}{Theorem}{Theorems}
\crefname{proposition}{Proposition}{Propositions}
\crefname{lemma}{Lemma}{Lemmas}
\crefname{corollary}{Corollary}{Corollaries}
\crefname{definition}{Definition}{Definitions}
\crefname{remark}{Remark}{Remarks}
\crefname{example}{Example}{Examples}
\crefname{section}{Section}{Sections}
\crefname{subsection}{Section}{Sections}
\Crefname{theorem}{Theorem}{Theorems}
\Crefname{proposition}{Proposition}{Propositions}
\Crefname{lemma}{Lemma}{Lemmas}
\Crefname{corollary}{Corollary}{Corollaries}
\Crefname{definition}{Definition}{Definitions}
\Crefname{remark}{Remark}{Remarks}
\Crefname{example}{Example}{Examples}
\Crefname{section}{Section}{Sections}
\Crefname{subsection}{Section}{Sections}
\AtBeginDocument{%
}

\newcommand{\frakX}{\mathfrak{X}}
\newcommand{\frakJ}{\mathfrak{J}}
\newcommand{\cA}{\mathcal{A}}
\newcommand{\cB}{\mathcal{B}}
\newcommand{\cH}{\mathcal{H}}
\newcommand{\NN}{\mathbb{N}}
\newcommand{\RR}{\mathbb{R}}
\newcommand{\CC}{\mathbb{C}}
\newcommand{\one}{\bm{1}}
\newcommand{\balpha}{\boldsymbol{\alpha}}
\newcommand{\bbeta}{\boldsymbol{\beta}}
\newcommand{\bgamma}{\boldsymbol{\gamma}}

\newcommand{\ba}{\boldsymbol{a}}
\newcommand{\bb}{\boldsymbol{b}}
\newcommand{\bzero}{\boldsymbol{0}}
\newcommand{\be}{\boldsymbol{e}}
\newcommand{\bu}{\boldsymbol{u}}
\newcommand{\widealpha}{\widehat{\boldsymbol{\alpha}}}

\newcommand{\eps}{\boldsymbol{\varepsilon}}
\newcommand{\pr}{\operatorname{pr}}
\newcommand{\Span}{\operatorname{Span}}
\newcommand{\domn}{\operatorname{dom}}
\newcommand{\rank}{\operatorname{rank}}
\newcommand{\im}{\operatorname{Im}}

\newcommand{\grlex}{\mathrm{grlex}}
\newcommand{\lex}{\mathrm{lex}}

\title{Multivariate $P$- and $Q$-polynomial structures of generalized Johnson association schemes}

\author{Hirotake Kurihara\thanks{Department of Applied Science, Yamaguchi University,
2-16-1 Tokiwadai, Ube 755-8611, Japan.
Email: \href{mailto:kurihara-hiro@yamaguchi-u.ac.jp}{\texttt{kurihara-hiro@yamaguchi-u.ac.jp}}.}}
\date{}

\begin{document}
\maketitle

\begin{abstract}
Let $\mathfrak X$ be a commutative association scheme of class $d$ on $X$. We construct the generalized Johnson association scheme $\mathfrak J_{n,h}(\mathfrak X)$ on partial maps from $[n]$ to $X$ with domains of size $h$. For $1\le h\le n-1$, we prove that $\mathfrak J_{n,h}(\mathfrak X)$ has $(d+1)$-variate $P$- and $Q$-polynomial structures.
In particular, its multivariate $Q$-polynomial property answers affirmatively a problem posed for
generalized Johnson association schemes arising from finite Gelfand pairs.
We also determine the primitive idempotents and the two eigenmatrices,
expressing the first eigenmatrix as a product of an Eberlein polynomial and a multivariate
Krawtchouk polynomial, and the second eigenmatrix as a product of a Hahn polynomial and
a multivariate Krawtchouk polynomial.
For the multivariate factors, we use Mizukawa--Tanaka
$(d+1,2d+2)$-hypergeometric functions.
In an appendix, we identify the normalized columns of the second eigenmatrix with the known spherical functions in the finite Gelfand pair case.
We also discuss the orthogonality relations and the imprimitivity arising from partial maps
with the same domain, and explicitly describe the multivariate $P$- and $Q$-polynomial structures
inherited by the blocks and the quotient.
\end{abstract}



\section{Introduction}\label{sec:introduction}

Throughout this paper, let $\NN=\{0,1,2,\ldots\}$.
For $n\in\NN$, write $[n]=\{1,2,\ldots,n\}$, and set $[0]=\varnothing$.
Let $\frakX=(X,\mathcal R,\{0,1,\ldots ,d\})$ be a commutative
association scheme of class $d$, and put $v=|X|$.
For terminology on association schemes, see \cref{subsec:association-schemes}.
Let $n$ be a positive integer and $0\le h\le n$.
A \emph{partial map} $\theta\colon[n]\rightharpoonup X$ is a map whose domain is a subset of $[n]$
and whose values lie in $X$. Denote its domain by $\domn\theta$, and define
\[
 \Theta
 =\bigl\{\theta\colon[n]\rightharpoonup X:|\domn\theta|=h\bigr\}.
\]
Thus, $\Theta$ is the set of all partial maps whose domains have size $h$.
In particular, $|\Theta|=\binom nh v^h$.
For $\theta,\eta\in\Theta$, define
\begin{equation}\label{eq:t-definition}
 t(\theta,\eta)=h-|\domn\theta\cap\domn\eta|.
\end{equation}
Also, for $0\le i\le d$, define
\[
\alpha_i(\theta,\eta)
 =\bigl|\{a\in\domn\theta\cap\domn\eta:(\theta(a),\eta(a))\in R_i\}\bigr|.
\]
Since $\{R_0,\ldots,R_d\}$ is a partition of $X\times X$, we have
\begin{equation}\label{eq:alpha-sum}
 \sum_{i=0}^{d}\alpha_i(\theta,\eta)=h-t(\theta,\eta).
\end{equation}

For $\balpha=(\alpha_1,\ldots,\alpha_d)\in\NN^d$, write
$|\balpha|=\sum_{i=1}^{d}\alpha_i$, and use $(t,\balpha)$ as an abbreviation for
$(t,\alpha_1,\ldots,\alpha_d)\in\NN^{d+1}$.
Define the index set
\begin{equation}\label{eq:D}
 \Lambda
 =\bigl\{(t,\balpha)\in\NN^{d+1}:
 0\le t\le n-h,\quad t+|\balpha|\le h\bigr\}.
\end{equation}

\begin{definition}\label[definition]{def:generalized-relation}
Define $\mathcal{S}\colon \Theta\times\Theta\longrightarrow \Lambda$ by
\[
 \mathcal{S}(\theta,\eta)
 =\bigl(t(\theta,\eta),
       \alpha_1(\theta,\eta),\ldots,\alpha_d(\theta,\eta)\bigr).
\]
For $(t,\balpha)\in \Lambda$, put $R_{t,\balpha}
 =\mathcal{S}^{-1}(t,\balpha)$, and let $A_{t,\balpha}$ be its adjacency matrix.
\end{definition}

Put $\frakJ_{n,h}(\frakX)=\bigl(\Theta,\mathcal S,\Lambda\bigr)$.
The purpose of this paper is to prove the following.
\begin{maintheorem}
The following hold.
\begin{enumerate}[label=\textup{(\roman*)},ref=\textup{(\roman*)}]
\item
$\frakJ_{n,h}(\frakX)$ is a commutative association scheme.
\textup{(\cref{thm:scheme} \ref{item:scheme-association})}
\item
Suppose that $1\le h\le n-1$. Then $\frakJ_{n,h}(\frakX)$ has a $(d+1)$-variate
$P$-polynomial structure. \textup{(\cref{thm:scheme} \ref{item:scheme-P})}
\item
Suppose that $1\le h\le n-1$. Then $\frakJ_{n,h}(\frakX)$ has a $(d+1)$-variate
$Q$-polynomial structure. \textup{(\cref{thm:Q})}
\end{enumerate}
\end{maintheorem}

The terminology in this theorem is defined in \cref{sec:prelim}, and more detailed
statements are given in \cref{thm:scheme,thm:Q}.
We henceforth call $\frakJ_{n,h}(\frakX)$ the \emph{generalized Johnson association scheme}
obtained from $\frakX$. We call $\frakX$ the \emph{base association scheme}
of $\frakJ_{n,h}(\frakX)$.

An early example of this construction is the \emph{nonbinary Johnson association scheme}.
Its points are words of constant weight over an alphabet with a distinguished zero symbol,
and its relations are determined by the size of the intersection of the two supports
and the number of coordinates on which the letters agree within that intersection.
In the present paper, this is the case where the base association scheme is the association scheme of a complete graph.
This construction appears in Dunkl's \cite{Dunkl1976} study of wreath product actions
and an addition theorem for Krawtchouk polynomials.
Tarnanen--Aaltonen--Goethals \cite{TarnanenAaltonenGoethals1985} gave an elementary
derivation of the eigenvalues of the nonbinary Johnson association scheme.

The generalization to an arbitrary finite Gelfand pair $(G,H)$ is due to
Ceccherini-Silberstein--Scarabotti--Tolli \cite{CST2006}.
They identify the space of partial maps into $X=G/H$ with
\[
 (G\wr\mathfrak S_n)\big/
 \bigl((H\wr\mathfrak S_h)\times(G\wr\mathfrak S_{n-h})\bigr),
\]
and call this space a generalized Johnson association scheme.
They also determine the multiplicity-free decomposition of the permutation representation
and the spherical functions, giving an expression that combines Hahn polynomials with
a symmetrization of the spherical functions of the base Gelfand pair
\cite{CST2006}.
Mishra--Srinivasan \cite{MishraSrinivasan2015} realize this space as the
rank-$h$ part of a generalized Boolean algebra and give combinatorial
derivations of the decomposition and eigenvalue formulas.

The construction is also related to the exponentiation of coherent
configurations; see Section~4 in \cite{CaiGuoGavrilyukPonomarenko2025}.
Here we use products of adjacency matrices and orthogonal projections
to determine the eigenmatrices and prove the multivariate $P$- and
$Q$-polynomial properties for every commutative base association scheme.

The multivariate $P$- and $Q$-polynomial properties describe adjacency matrices and
primitive idempotents as polynomials in several generators.
The relationship between multivariate orthogonal polynomials and association schemes
has long been studied. For example, Mizukawa--Tanaka \cite{MizukawaTanaka2004}
expressed the eigenmatrices of symmetric tensor powers of character algebras in terms of
multivariate hypergeometric functions.
As polynomial structures of association schemes, bivariate versions were formulated by
Bernard--Cramp\'e--Poulain d'Andecy--Vinet--Zaimi \cite{BCPVZ2024}.
Subsequently, Bannai--Kurihara--Zhao--Zhu \cite{BKZZ2025} gave definitions and criteria
in an arbitrary number of variables using general monomial orders. We use the latter definition in this paper.

For generalized Johnson association schemes, Theorem~4.5 in \cite{BKZZ2025} establishes the
multivariate $P$-polynomial property in the finite Gelfand pair case.
As for the $Q$-polynomial property, Cramp\'e--Vinet--Zaimi--Zhang \cite{CVZZ2024}
gave a bivariate $Q$-polynomial structure for the nonbinary Johnson association scheme under the condition
$n\ge2h-1$, and Bannai--Kurihara--Zhao--Zhu \cite{BKZZ2024} removed this parameter
restriction under the definition using general monomial orders.
For the general construction from finite Gelfand pairs, Remark~10 in \cite{BKZZ2025}
lists the multivariate $Q$-polynomial property as an open problem.
\Cref{thm:Q} answers this problem affirmatively and further shows that the base may be
any commutative association scheme, without assuming a group action.
We do not assume that the base association scheme itself is $P$-polynomial or $Q$-polynomial.
Moreover, for $1\le h \le n$,
\cref{prop:GJ-schurian} shows that a non-Schurian base association scheme gives a non-Schurian generalized Johnson association scheme.

We prove the multivariate $P$-polynomial property using $A_{1,\bzero}$
and $A_{0,\be_i}$ $(1\le i\le d)$ as generators, where $\be_1,\ldots,\be_d$
are the standard unit vectors of $\NN^d$.
Here \cref{prop:scheme-from-generators} gives a general criterion for recognizing
commutative association schemes from products of the generators, without first
computing all intersection numbers. It applies to other partitions satisfying
the same triangularity conditions, independently of the present construction.
We then combine the primitive idempotents of $\frakX$ with those of Johnson
association schemes to construct the primitive idempotents of
$\frakJ_{n,h}(\frakX)$ and obtain its two eigenmatrices in
\cref{thm:factorization-main}.
The first eigenmatrix of $\frakJ_{n,h}(\frakX)$ is a product of an Eberlein polynomial and a first eigenvalue of
an extension, and the second eigenmatrix is a product of a Hahn polynomial and a second eigenvalue
of an extension. The extension factors are the multivariate Krawtchouk polynomials of
Mizukawa--Tanaka.
To prove the multivariate $Q$-polynomial property, we use the formula expressing
Krein parameters in terms of the second eigenmatrix to compute those of
$\frakJ_{n,h}(\frakX)$. We then verify the required triangularity and
nonvanishing conditions using the triangularity of the extension Krein
parameters and the orthogonality of Hahn polynomials.

The two appendices collect supplements obtained from these calculations.
In the finite Gelfand pair case, the columns of the second eigenmatrix
correspond to spherical functions.
Appendix~\ref{sec:kernel} makes this correspondence explicit and compares the
resulting formula with Theorem~6.18 in \cite{CST2006}.
For the normalized columns, we give the weight, squared norms, and the orthogonality and dual orthogonality relations.
The parameters of the Hahn polynomial depend on $|\bbeta|$,
where $(\bbeta,k)$ indexes a primitive idempotent, while the
length parameter of the multivariate Krawtchouk polynomial is $h-t$.
This dependence resembles Tratnik's \cite{Tratnik1991} product construction.
This feature can be viewed as an extension of the bivariate polynomials of the nonbinary
Johnson association scheme \cite{CVZZ2024} to an arbitrary commutative base association scheme.
In Appendix~\ref{sec:imprimitive}, we examine the imprimitivity arising from partial maps
with the same domain. The block association scheme is the extension of $\frakX$ of length $h$, and the quotient association scheme is
$J(n,h)$. We explicitly describe their respective $d$-variate and univariate
$P$- and $Q$-polynomial structures.
These polynomial structures agree with the inheritance formulas in
Theorems~4.6 and~4.7 in \cite{HigashitaniKurihara2026}.

\Cref{sec:prelim} recalls the definitions of association schemes and
multivariate polynomial structures, and states the criteria used in the proofs.
\Cref{sec:scheme-examples} records the formulas for Johnson association schemes
and extensions needed in the eigenmatrix calculation.
\Cref{sec:construction} proves the construction and the $P$-polynomial property of
generalized Johnson association schemes, and relates their Schurian property to that of the base association scheme.
\Cref{sec:primitive-idempotents_of_GJS} determines the primitive idempotents and
eigenmatrices, and proves the multivariate $Q$-polynomial property.
Appendices~\ref{sec:kernel} and~\ref{sec:imprimitive} treat, respectively, spherical
functions and orthogonality relations, and imprimitivity and the polynomial structures
of the blocks and the quotient.

\section{Association schemes and multivariate polynomial properties}\label{sec:prelim}

This section fixes notation for commutative association schemes, their
eigenmatrices, intersection numbers, and Krein parameters. We then recall
monomial orders and multivariate $P$- and $Q$-polynomial structures,
and state the criteria used to prove these properties in
\cref{sec:construction,sec:primitive-idempotents_of_GJS}.

\subsection{Commutative association schemes}\label{subsec:association-schemes}
For details of this subsection, see \cite{BannaiIto1984}.
For a finite set $X$, write $v=|X|$ and $\Delta_X=\{(x,x):x\in X\}$.
Define the transpose of a relation $R\subseteq X\times X$ by
$R^{\mathsf T}=\{(y,x):(x,y)\in R\}$.
Let $\mathcal I$ be a finite set and
$\mathcal R\colon X\times X\longrightarrow\mathcal I$ a surjection.
For $i\in\mathcal I$, put $R_i=\mathcal R^{-1}(i)$.
The \emph{adjacency matrix} $A_i$ of $R_i$ is the $0$--$1$ matrix with rows and
columns indexed by $X$ given by
\[
 (A_i)_{xy}=
 \begin{cases}
 1,&(x,y)\in R_i,\\
 0,&(x,y)\notin R_i.
 \end{cases}
\]
Let $I_X$ and $J_X$ be the identity matrix and the all-ones matrix on $X$,
respectively, and let $\one$ be the all-ones column vector.
Since $\{R_i\}_{i\in\mathcal I}$ is a partition of $X\times X$, we have
$\sum_{i\in\mathcal I}A_i=J_X$.

\begin{definition}\label[definition]{def:association-scheme}
The triple $\frakX=(X,\mathcal R,\mathcal I)$ is called a \emph{commutative association scheme} if the following conditions hold.
\begin{enumerate}[
 label=\textup{(AS\arabic*)},
 ref=\textup{(AS\arabic*)},
 leftmargin=3.6em]
\item\label{item:as-identity}
There exists $0\in\mathcal I$ such that $R_0=\Delta_X$, that is, $A_0=I_X$.
\item\label{item:as-transpose}
For each $i\in\mathcal I$, there exists $i'\in\mathcal I$ such that
$R_i^{\mathsf T}=R_{i'}$, that is, $A_i^{\mathsf T}=A_{i'}$.
\item\label{item:as-intersection}
For $i,j,\ell\in\mathcal I$ and $(x,y)\in R_\ell$, the number
\[
 \bigl|\{z\in X:(x,z)\in R_i,\ (z,y)\in R_j\}\bigr|
\]
is independent of the choice of $(x,y)$. 
Denote this number by $p_{ij}^{\ell}$.
Then we have
\[
 A_iA_j=\sum_{\ell\in\mathcal I}p_{ij}^{\ell}A_\ell
 \qquad(i,j\in\mathcal I).
\]
\item\label{item:as-commutativity}
We have $p_{ij}^{\ell}=p_{ji}^{\ell}$ for all $i,j,\ell\in\mathcal I$,
that is, $A_iA_j=A_jA_i$ for all $i,j\in\mathcal I$.
\end{enumerate}
The numbers $p_{ij}^{\ell}$ are called the \emph{intersection numbers}.
\end{definition}

If $i'=i$ for all $i\in\mathcal I$, then $\frakX$ is called a
\emph{symmetric association scheme}.
We call $d=|\mathcal I|-1$ the \emph{class} of $\frakX$.
Write the \emph{valency} of $R_i$ as
$k_i=\bigl|\{y\in X:(x,y)\in R_i\}\bigr|=p_{i\,i'}^0$.
By \ref{item:as-intersection}, this value is independent of the choice of
$x\in X$; in particular, $k_0=1$.
Moreover, \ref{item:as-transpose} and \ref{item:as-commutativity}, together
with $\sum_{r\in\mathcal I}A_r=J_X$, give
$A_i\one=k_i\one$, $\one^{\mathsf T}A_i=k_i\one^{\mathsf T}$,
and $k_{i'}=k_i$ for all $i\in\mathcal I$.
By \ref{item:as-transpose}--\ref{item:as-commutativity},
$\cA(\frakX)=\Span_{\CC}\{A_i:i\in\mathcal I\}$ is a
$(d+1)$-dimensional commutative algebra closed under transposition.
It is called the \emph{Bose--Mesner algebra} of $\frakX$.

Equip $\CC^X$ with the inner product
\[
 \langle f,g\rangle_X=\sum_{x\in X}\overline{f(x)}g(x),
\]
and write $\{e_y:y\in X\}$ for its standard orthonormal basis and $M^*$ for
the conjugate transpose of a matrix $M$.
The algebra $\cA(\frakX)$ is commutative and closed under $*$, so its \emph{primitive idempotents} $\{E_j\}_{j\in\mathcal J}$ are the orthogonal projections onto the
maximal common eigenspaces $U_j=E_j\CC^X$. They form a basis of $\cA(\frakX)$ and satisfy
\[
 E_jE_\ell=\delta_{j\ell}E_j,\qquad E_j^*=E_j,\qquad
 \sum_{j\in\mathcal J}E_j=I_X,\qquad E_0=\frac1vJ_X.
\]
Although $0\in\mathcal I$ and $0\in\mathcal J$ are different objects, we use
the same notation since they can be distinguished from the context.
Here $|\mathcal J|=d+1$, and the \emph{multiplicity} of $E_j$ is
$m_j=\rank E_j=\dim U_j$.
We will use the following formula:
$(E_j)_{x,y}=\langle E_je_x,E_je_y\rangle_X$.
Define the \emph{first eigenmatrix} $P$ and the \emph{second eigenmatrix} $Q$
by the coefficients between the two bases:
\begin{equation}\label{eq:base-PQ}
 A_i=\sum_{j\in\mathcal J}P_i(j)E_j,\qquad
 vE_j=\sum_{i\in\mathcal I}Q_j(i)A_i.
\end{equation}
Here $i\in\mathcal I$ and $j\in\mathcal J$. In particular,
\begin{equation}\label{eq:PQ-zero}
 P_i(0)=k_i\quad(i\in\mathcal I),\qquad
 Q_j(0)=m_j\quad(j\in\mathcal J).
\end{equation}

The entrywise product of matrices is called the \emph{Hadamard product} and is denoted by $M\circ N$.
Note that $A_i\circ A_r=\delta_{ir}A_i$ for all $i,r\in\mathcal I$.
Hence $\cA(\frakX)$ is also closed under the Hadamard product.
Define the \emph{Krein parameters} $q_{jk}^{\ell}$ $(j,k,\ell\in\mathcal J)$ by
\begin{equation}\label{eq:Krein-definition}
 (vE_j)\circ(vE_k)=\sum_{\ell\in\mathcal J}q_{jk}^{\ell}\,vE_\ell.
\end{equation}
The $q_{jk}^{\ell}$ are known to be nonnegative real numbers.
Note that, for all $j,\ell\in\mathcal J$ and $i\in\mathcal I$,
\begin{align}
 Q_j(i')&=\overline{Q_j(i)},\label{eq:Q-transpose}\\
 \frac{P_i(j)}{k_i}
 &=\frac{Q_j(i')}{m_j},\label{eq:PQ-entry-duality}\\
 \sum_{i\in\mathcal I}k_i\overline{Q_j(i)}Q_\ell(i)
 &=vm_j\delta_{j\ell}.\label{eq:Q-orthogonality}
\end{align}
We henceforth regard the $j$th column of the second eigenmatrix $Q$ as a function
$Q_j$ on $\mathcal I$.
By \eqref{eq:Krein-definition}, their entrywise product can be written as
\[
 Q_jQ_k=\sum_{\ell\in\mathcal J}q_{jk}^{\ell}Q_\ell.
\]
Thus, by \eqref{eq:Q-orthogonality} and the reality of $q_{jk}^{\ell}$,
for $j,k,\ell\in\mathcal J$ we have
\begin{equation}\label{eq:Krein-eigenmatrix-formula2}
 q_{jk}^{\ell}=\frac{1}{vm_\ell}
 \sum_{i\in\mathcal I}k_i\overline{Q_j(i)Q_k(i)}Q_\ell(i).
\end{equation}

\subsection{Monomial orders and multivariate polynomial properties}\label{subsec:orders}
For details of this subsection, see \cite{BKZZ2025,HigashitaniKurihara2026}.
For $r$ variables $x_1,\ldots,x_r$, identify
$\balpha=(\alpha_1,\ldots,\alpha_r)\in\NN^r$ with the monomial
$x^{\balpha}=x_1^{\alpha_1}\cdots x_r^{\alpha_r}$.

A total order $\preceq$ on $\NN^r$ is a \emph{monomial order} if
$\bzero\preceq\balpha$ for every $\balpha\in\NN^r$ and
\[
 \balpha\preceq\bbeta
 \quad\Longrightarrow\quad
 \balpha+\bgamma\preceq\bbeta+\bgamma
 \qquad(\balpha,\bbeta,\bgamma\in\NN^r).
\]
The \emph{lexicographic order} $\preceq_{\lex}$ compares components from left to right,
using the first component at which they differ.
The \emph{graded lexicographic order} $\preceq_{\grlex}$ first compares total degrees
$|\balpha|=\sum_i\alpha_i$ and breaks ties by $\preceq_{\lex}$.

Let $1\le s<r$, and regard $\NN^r=\NN^s\times\NN^{r-s}$.

\begin{definition}\label[definition]{def:elimination-block-order}
A monomial order $\preceq$ is of \emph{$s$-elimination type} if
\[
 (\bzero,\bbeta_2)\prec(\balpha,\bbeta_1)
\]
for every $\balpha\in\NN^s\setminus\{\bzero\}$ and
$\bbeta_1,\bbeta_2\in\NN^{r-s}$.
If, in addition,
\[
 (\balpha_1,\bzero)\prec(\balpha_2,\bzero)
 \quad\Longrightarrow\quad
 (\balpha_1,\bbeta_1)\prec(\balpha_2,\bbeta_2)
\]
holds for all $\bbeta_1,\bbeta_2\in\NN^{r-s}$, then $\preceq$ is said to be
of \emph{$s$-block type}.
\end{definition}

Every order of $s$-block type is of $s$-elimination type; the lexicographic
order is an example of both, with respect to this decomposition
\cite{HigashitaniKurihara2026}.
A subset $D\subseteq\NN^r$ is \emph{downward closed} if $\balpha\in D$ and
$0\le\beta_i\le\alpha_i$ for all $i$ imply $\bbeta\in D$.
In particular, a nonempty downward closed set contains $\bzero$.
Let $\eps_1,\ldots,\eps_r$ be the standard unit vectors of $\NN^r$.

\begin{definition}\label[definition]{def:multivariate-polynomial}
Let $D\subseteq\NN^r$ be a finite downward closed subset.
Suppose that $\eps_1,\ldots,\eps_r\in D$.
Suppose that the adjacency matrices of a commutative association scheme are
indexed as $\{A_{\balpha}\}_{\balpha\in D}$.
Write $A^{\balpha}=A_{\eps_1}^{\alpha_1}\cdots A_{\eps_r}^{\alpha_r}$.
The association scheme is said to be \emph{$r$-variate $P$-polynomial} on $D$ with respect to
a monomial order $\preceq$ if the following two conditions hold.
\begin{enumerate}[
 label=\textup{(P\arabic*)},
 ref=\textup{(P\arabic*)},
 leftmargin=3.2em]
\item\label{item:p-polynomial-representation}
For each $\balpha\in D$, there is a polynomial
$v_{\balpha}\in\CC[x_1,\ldots,x_r]$ with
$A_{\balpha}=v_{\balpha}(A_{\eps_1},\ldots,A_{\eps_r})$.
The leading term of $v_{\balpha}$ with respect to $\preceq$ is a nonzero
constant multiple of $x^{\balpha}$, and the exponent $\bbeta$ of every monomial
$x^{\bbeta}$ occurring in $v_{\balpha}$ belongs to $D$.
\item\label{item:p-polynomial-triangularity}
For every $i$ and $\balpha\in D$, we have
\[
 A_{\eps_i}A^{\balpha}
 \in\Span\{A^{\bbeta}:\bbeta\in D,\
                         \bbeta\preceq\balpha+\eps_i\}.
\]
\end{enumerate}
We call $v_{\balpha}$ an \emph{associated polynomial}.
\end{definition}

\begin{definition}\label[definition]{def:multivariate-Q-polynomial}
Let $D^*\subseteq\NN^{r^*}$ be a finite downward closed subset containing
$\eps_1,\ldots,\eps_{r^*}$, and let $\preceq^*$ be a monomial order on
$\NN^{r^*}$.
Suppose that the primitive idempotents of a commutative association scheme
are indexed as $\{E_{\balpha}\}_{\balpha\in D^*}$, with
$E_{\bzero}=|X|^{-1}J_X$, and put $F_{\balpha}=|X|E_{\balpha}$.
The scheme is \emph{$r^*$-variate $Q$-polynomial} on $D^*$ with respect to
$\preceq^*$ if conditions \ref{item:p-polynomial-representation} and
\ref{item:p-polynomial-triangularity} of \cref{def:multivariate-polynomial}
hold with $(D,r,\preceq,A_{\balpha})$ replaced by
$(D^*,r^*,\preceq^*,F_{\balpha})$ and all products interpreted
as Hadamard products. In particular, the constant polynomial $1$
is evaluated as $F_{\bzero}=J_X$.
The associated polynomials are denoted by $v_{\balpha}^*$.
\end{definition}

We first state equivalent conditions for the multivariate $P$-polynomial property
in terms of intersection numbers.
\begin{proposition}
\label[proposition]{prop:P-criterion}
Let $D\subseteq\NN^r$ be a finite downward closed subset containing
$\eps_1,\ldots,\eps_r$, and let $\preceq$ be a monomial order on $\NN^r$.
Suppose that the adjacency matrices of a commutative association scheme
$\frakX$ are indexed as $\{A_{\balpha}\}_{\balpha\in D}$, with $A_{\bzero}=I_X$.
Then the following three conditions are equivalent.
\begin{enumerate}[label=\textup{(\roman*)},ref=\textup{(\roman*)}]
\item\label{item:p-criterion-polynomial}
$\frakX$ is $r$-variate $P$-polynomial on $D$ with respect to $\preceq$.
\item\label{item:p-criterion-intersection}
For all $1\le i\le r$ and $\balpha,\bbeta\in D$, we have
\[
 p_{\eps_i\balpha}^{\bbeta}\ne0
 \quad\Longrightarrow\quad
 \bbeta\preceq\balpha+\eps_i.
\]
Moreover, if $\balpha+\eps_i\in D$, then
\[
 p_{\eps_i\balpha}^{\balpha+\eps_i}\ne0.
\]
\item\label{item:p-criterion-intersection-triangularity}
For arbitrary $\balpha,\bbeta,\bgamma\in D$, we have
\[
 p_{\balpha\bbeta}^{\bgamma}\ne0
 \quad\Longrightarrow\quad
 \bgamma\preceq\balpha+\bbeta.
\]
Moreover, if $\balpha+\bbeta\in D$, then
\[
 p_{\balpha\bbeta}^{\balpha+\bbeta}\ne0.
\]
\end{enumerate}
\end{proposition}
\begin{proof}
The equivalence of \ref{item:p-criterion-polynomial} and
\ref{item:p-criterion-intersection} is proved by Proposition~2.15 in \cite{BKZZ2025}.
The implication from \ref{item:p-criterion-polynomial} to
\ref{item:p-criterion-intersection-triangularity} follows from Lemma~2.16(ii) in
\cite{HigashitaniKurihara2026}.
Since \ref{item:p-criterion-intersection} is a special case of
\ref{item:p-criterion-intersection-triangularity},
\ref{item:p-criterion-intersection-triangularity} implies
\ref{item:p-criterion-intersection}.
Thus the three conditions are equivalent.
\end{proof}

The following criterion can be used to verify that a candidate partition
forms a commutative association scheme without first computing all intersection
numbers. 
It suffices to verify that the selected generators form a set closed under transposition, commute with each other, and satisfy the triangular conditions in
\cref{prop:P-criterion} \ref{item:p-criterion-intersection}.

\begin{proposition}
\label[proposition]{prop:scheme-from-generators}
Let $r$ be a positive integer, $X$ a finite set, and $D\subseteq\NN^r$ a finite
downward closed set containing $\eps_1,\ldots,\eps_r$.
Let $\preceq$ be a monomial order on $\NN^r$.
For a surjection
\[
 \mathcal R\colon X\times X\longrightarrow D,
\]
put $R_{\balpha}=\mathcal R^{-1}(\balpha)$ and let $A_{\balpha}$ be its relation matrix.
Assume the following.
\begin{enumerate}[label=\textup{(\roman*)},ref=\textup{(\roman*)}]
\item\label{item:generator-partition}
$R_{\bzero}=\Delta_X$, and $\{A_{\eps_1},\ldots,A_{\eps_r}\}$ is closed under transposition.
\item\label{item:generator-commutativity}
$A_{\eps_1},\ldots,A_{\eps_r}$ commute with each other.
\item\label{item:generator-triangularity}
For every $1\le i\le r$ and $\balpha\in D$, we can write
\begin{equation}\label{eq:generator-product-criterion}
 A_{\eps_i}A_{\balpha}
 =\sum_{\bbeta\in D}c_{i,\balpha}^{\bbeta}A_{\bbeta},
\end{equation}
where
\[
 c_{i,\balpha}^{\bbeta}\ne0
 \quad\Longrightarrow\quad
 \bbeta\preceq\balpha+\eps_i.
\]
Moreover, if $\balpha+\eps_i\in D$, then $c_{i,\balpha}^{\balpha+\eps_i}\ne0$.
\end{enumerate}
Then $(X,\mathcal R,D)$ is a commutative association scheme and is $r$-variate
$P$-polynomial on $D$ with respect to $\preceq$.
In particular, the coefficient $c_{i,\balpha}^{\bbeta}$ in
\eqref{eq:generator-product-criterion} equals the intersection number
$p_{\eps_i\balpha}^{\bbeta}$ of this association scheme.
\end{proposition}

\begin{proof}
Since $\mathcal R$ is surjective, the nonempty relations $R_{\balpha}$ form a
partition of $X\times X$.
In particular, $A_{\balpha}$ are linearly independent and
$\sum_{\balpha\in D}A_{\balpha}=J_X$.
Let $\cA=\CC[A_{\eps_1},\ldots,A_{\eps_r}]$ be the unital matrix algebra
generated by $A_{\eps_1},\ldots,A_{\eps_r}$, and let
$\cB=\Span_{\CC}\{A_{\balpha}:\balpha\in D\}$ be the vector space spanned
by $\{A_{\balpha}: \balpha\in D\}$.
By condition~\ref{item:generator-triangularity}, $\cB$ is preserved by left
multiplication by each generator. Hence $\cA\subseteq\cB$.

We prove the reverse inclusion by induction with respect to $\preceq$.
We have $A_{\bzero}=I_X\in\cA$.
Let $\balpha\ne\bzero$, choose $i$ with $\alpha_i>0$, and put
$\bgamma=\balpha-\eps_i$.
Since $D$ is downward closed, $\bgamma\in D$, and $\bgamma\prec\balpha$.
Applying condition~\ref{item:generator-triangularity} to $\bgamma$ gives
\[
 A_{\eps_i}A_{\bgamma}
 =c_{i,\bgamma}^{\balpha}A_{\balpha}
  +\sum_{\substack{\bbeta\in D\\\bbeta\prec\balpha}}
   c_{i,\bgamma}^{\bbeta}A_{\bbeta},
 \qquad
 c_{i,\bgamma}^{\balpha}\ne0.
\]
By the induction hypothesis, $A_{\bgamma}$ and all the $A_{\bbeta}$ occurring
in the sum belong to $\cA$. Solving this equality for $A_{\balpha}$ therefore
gives $A_{\balpha}\in\cA$.
Thus $\cB\subseteq\cA$, and hence $\cA=\cB$.
By condition~\ref{item:generator-commutativity}, this is a commutative algebra
closed under ordinary matrix multiplication.

By \ref{item:generator-partition}, the relation matrices indexed by
$\eps_1,\ldots,\eps_r$ are closed under transposition. Since $\cA=\cB$,
we have $\cB^{\mathsf T}=\cB$.
Also, since $A_{\balpha}\circ A_{\bbeta}=\delta_{\balpha\bbeta}A_{\balpha}$
and $\sum_{\balpha\in D}A_{\balpha}=J_X$, the space $\cB$ is an algebra
under the Hadamard product as well.
Transposition is an automorphism of $\cB$ with respect to the Hadamard product,
and the pairwise disjoint nonzero $0$--$1$ matrices $A_{\balpha}$ are the
primitive idempotents of this Hadamard algebra.
Hence transposition permutes $\{A_{\balpha}:\balpha\in D\}$.

Since $\cB$ is closed under multiplication, for arbitrary $\balpha,\bbeta\in D$
we can write uniquely
\[
 A_{\balpha}A_{\bbeta}
 =\sum_{\bgamma\in D}p_{\balpha\bbeta}^{\bgamma}A_{\bgamma}.
\]
By \ref{item:as-intersection}, the $p_{\balpha\bbeta}^{\bgamma}$ are
precisely the intersection numbers.
Together with $A_{\bzero}=I_X$, closure under transposition, and commutativity
of $\cB$, this verifies all the conditions in \cref{def:association-scheme}.
Finally, condition~\ref{item:generator-triangularity} and \cref{prop:P-criterion}
show that this association scheme is $r$-variate $P$-polynomial.
\end{proof}

For the multivariate $Q$-polynomial property, we use the following criterion.

\begin{proposition}
\label[proposition]{prop:Q-criterion}
Let $D^*\subseteq\NN^{r^\ast}$ be a finite downward closed subset containing
$\eps_1,\ldots,\eps_{r^\ast}$, and let $\preceq^\ast$ be a monomial order on
$\NN^{r^\ast}$.
Suppose that the primitive idempotents of a commutative association scheme
$\frakX$ are indexed as $\{E_{\balpha}\}_{\balpha\in D^*}$, with
$E_{\bzero}=|X|^{-1}J_X$.
Then the following three conditions are equivalent.
\begin{enumerate}[label=\textup{(\roman*)},ref=\textup{(\roman*)}]
\item\label{item:q-criterion-polynomial}
$\frakX$ is $r^\ast$-variate $Q$-polynomial on $D^*$ with respect to $\preceq^\ast$.
\item\label{item:q-criterion-krein}
For all $1\le i\le r^\ast$ and $\balpha,\bbeta\in D^*$, we have
\[
 q_{\eps_i\balpha}^{\bbeta}\ne0
 \quad\Longrightarrow\quad
 \bbeta\preceq^\ast\balpha+\eps_i.
\]
Moreover, if $\balpha+\eps_i\in D^*$, then
$q_{\eps_i\balpha}^{\balpha+\eps_i}\ne0$.
\item \label{item:q-criterion-krein-triangularity}
For arbitrary $\balpha,\bbeta,\bgamma\in D^*$, we have
\[
 q_{\balpha\bbeta}^{\bgamma}\ne0
 \quad\Longrightarrow\quad
 \bgamma\preceq^\ast\balpha+\bbeta.
\]
Moreover, if $\balpha+\bbeta\in D^*$, then
\[
 q_{\balpha\bbeta}^{\balpha+\bbeta}\ne0.
\]
\end{enumerate}
\end{proposition}
\begin{proof}
The equivalence of \ref{item:q-criterion-polynomial} and
\ref{item:q-criterion-krein} follows from Proposition~2.16 in
\cite{BKZZ2025}.
The implication from \ref{item:q-criterion-polynomial} to
\ref{item:q-criterion-krein-triangularity} follows from Lemma~2.18(ii) in
\cite{HigashitaniKurihara2026}.
Finally, \ref{item:q-criterion-krein-triangularity} implies
\ref{item:q-criterion-krein}, which completes the proof.
\end{proof}

\section{Johnson association schemes and extensions}\label{sec:scheme-examples}

The eigenmatrices of generalized Johnson association schemes combine factors
from Johnson association schemes and extensions. We give their formulas,
together with the two lemmas used to prove the multivariate
$Q$-polynomial property.

\subsection{Johnson association schemes}\label{subsec:Johnson}
This subsection summarizes basic facts about Johnson association schemes.
For detailed proofs, see \cite{BannaiIto1984}.
Let $N,H$ be nonnegative integers satisfying $0\le H\le N$, and put
$T=\min\{H,N-H\}$.
Write $\binom{[N]}{H}=\{A\subseteq[N]:|A|=H\}$ and define the \emph{Johnson distance} by
\begin{equation}\label{eq:Johnson-distance}
 \delta(A,B)=H-|A\cap B|
 \qquad\bigl(A,B\in\tbinom{[N]}{H}\bigr).
\end{equation}
Then $J(N,H)=\bigl(\binom{[N]}{H},\delta,\{0,1,\ldots,T\}\bigr)$ is 
a symmetric association scheme,
and is called the \emph{Johnson association scheme}.
Write $A_t^{J(N,H)}$ for the adjacency matrix of the relation $\delta(A,B)=t$.
Its valency is
\begin{equation}\label{eq:Johnson-valency}
 \nu_t^{N,H}=\binom Ht\binom{N-H}{t}.
\end{equation}
In particular, the degree of the distance-$1$ Johnson graph is $H(N-H)$.
Throughout, a binomial coefficient is taken to be $0$ if its lower index is
negative or exceeds its upper index.

Index the primitive idempotents $E_k^{J(N,H)}$ by $0\le k\le T$ in the
order specified by the eigenmatrix formulas below. Their multiplicities are
\begin{equation}\label{eq:johnson-dimension}
 d_{k}^{N}:=\rank E_k^{J(N,H)}
 =\binom Nk-\binom N{k-1}
 \qquad(0\le k\le T).
\end{equation}
In particular, $d_{0}^{N}=1$ and
$E_0^{J(N,H)}=\binom NH^{-1}J_{\binom{[N]}{H}}$.

The entries of the first and second eigenmatrices of the Johnson association
scheme are known to be expressed by the following orthogonal polynomials.
For $0\le t,k\le T$, put
\begin{equation}\label{eq:Eberlein}
 \mathsf E_t(N,H;k)
 =\sum_{u=0}^t(-1)^u
 \binom ku\binom{H-k}{t-u}\binom{N-H-k}{t-u}.
\end{equation}
This is an \emph{Eberlein (dual Hahn) polynomial} of degree $t$ in $k(k-N-1)$.
Also, let the \emph{Pochhammer symbol} be
\[
 (a)_0=1,\qquad (a)_s=\prod_{j=0}^{s-1}(a+j)\quad(s\ge1),
\]
and define
\begin{equation}\label{eq:Hahn-eigenvalue}
 \mathsf H_k(N,H;t)
 =d_{k}^{N}\sum_{s=0}^k
 \frac{(-k)_s(k-N-1)_s(-t)_s}
 {(H-N)_s(-H)_s\,s!}.
\end{equation}
The denominator does not vanish in the range $s\le k\le T$.
In particular,
\begin{equation}\label{eq:Hahn-zero-one}
 \mathsf H_k(N,H;0)=d_{k}^{N},\qquad
 \mathsf H_1(N,H;t)=(N-1)\left(1-\frac{N}{H(N-H)}t\right)
 \quad(1\le H\le N-1).
\end{equation}
The polynomial $\mathsf H_k(N,H;x)$ is a real polynomial of degree $k$ in $x$,
called a \emph{Hahn polynomial}.
In particular, $\mathsf H_0(N,H;t)=1$.
In the convention of \eqref{eq:base-PQ}, the eigenmatrices of $J(N,H)$ are
\begin{equation}\label{eq:Johnson-eigenmatrices}
 P_t^{J(N,H)}(k)=\mathsf E_t(N,H;k),\qquad
 Q_k^{J(N,H)}(t)=\mathsf H_k(N,H;t)
 \quad(0\le t,k\le T).
\end{equation}
In particular, $P_t^{J(N,H)}(0)=\nu_t^{N,H}$ and $Q_k^{J(N,H)}(0)=d_{k}^{N}$.
Thus the entries of $E_k^{J(N,H)}$ are given by
\begin{equation}\label{eq:johnson-projector}
 \bigl(E_k^{J(N,H)}\bigr)_{A,B}
 =\frac{1}{\binom NH}
   \mathsf H_k\bigl(N,H;\delta(A,B)\bigr)
 \qquad(A,B\in\binom{[N]}{H}).
\end{equation}
Note that, for $T\ge1$, $J(N,H)$ is a univariate $P$- and $Q$-polynomial association scheme.

By symmetry and \eqref{eq:PQ-entry-duality}, we have
\begin{equation}\label{eq:Johnson-PQ-relation}
 \mathsf E_t(N,H;k)
 =\frac{\nu_t^{N,H}}{d_{k}^{N}}\mathsf H_k(N,H;t).
\end{equation}
We give the orthogonality relation for Hahn polynomials that will be used later.
By \eqref{eq:Q-orthogonality} and \eqref{eq:Johnson-eigenmatrices}, we have
\begin{equation}\label{eq:E-orthogonality-Krein}
 \sum_{t=0}^T\nu_t^{N,H}
 \mathsf H_k(N,H;t)\mathsf H_\ell(N,H;t)
 =\binom NH d_{k}^{N}\delta_{k\ell}
 \qquad(0\le k,\ell\le T).
\end{equation}
Since $\mathsf H_k(N,H;t)$ is a real polynomial of degree $k$ in $t$, it is
orthogonal to every polynomial of degree less than $k$ with respect to the weight
$\nu_t^{N,H}$.
If $H=0$ or $H=N$, then $T=0$, and this formula holds with $t=k=\ell=0$.

\begin{lemma}\label[lemma]{lem:Johnson-quotient-degree}
Let $0\le k\le T$ and let $p\in\RR[x]$ be a nonzero polynomial.
Then
\[
 \sum_{t=0}^T\nu_t^{N,H}\mathsf H_k(N,H;t)p(t)\ne0
 \quad\Longrightarrow\quad \deg p\ge k.
\]
Moreover, if $\deg p=k$, then
\[
 \sum_{t=0}^T\nu_t^{N,H}\mathsf H_k(N,H;t)p(t)\ne0.
\]
\end{lemma}
\begin{proof}
If $\deg p<k$, the sum vanishes by \eqref{eq:E-orthogonality-Krein},
since the Hahn polynomials of degrees less than $k$ span all polynomials
of degree less than $k$.
If $\deg p=k$, write uniquely
$p(x)=\sum_{j=0}^k c_j\mathsf H_j(N,H;x)$ with $c_k\ne0$.
Then \eqref{eq:E-orthogonality-Krein} gives
\[
 \sum_{t=0}^T\nu_t^{N,H}\mathsf H_k(N,H;t)p(t)
 =\binom NH d_{k}^{N}c_k\ne0.
\]
\end{proof}

\subsection{Extensions and multivariate Krawtchouk polynomials}\label{sec:factorization-prelim}

Let $\frakX=(X,\mathcal R,\{0,1,\ldots,d\})$ be a commutative association
scheme of class $d$, put $v=|X|$, and let $L\in\NN$.
Write $A_0,\ldots,A_d$ and $E_0,\ldots,E_d$ for its adjacency matrices and
primitive idempotents, with $E_0=v^{-1}J_X$.
Set
\[
 D_L^{\mathrm{ext}}=\{\boldsymbol\nu\in\NN^d:|\boldsymbol\nu|\le L\},
 \qquad \nu_0=L-|\boldsymbol\nu|\quad(\boldsymbol\nu\in D_L^{\mathrm{ext}}).
\]
We take zero-fold Cartesian powers to be singletons, empty sums to be $0$,
empty products and empty multinomial coefficients to be $1$, and the empty
tensor product to be the $1\times1$ matrix $(1)$.
For $\mathbf x,\mathbf y\in X^L$ and $0\le i\le d$, put
\[
 \tau_i(\mathbf x,\mathbf y)
 =\bigl|\{\ell\in[L]:(x_\ell,y_\ell)\in R_i\}\bigr|,
\]
so that $\sum_{i=0}^d\tau_i(\mathbf x,\mathbf y)=L$.

\begin{definition}\label[definition]{def:extension}
Define
\[
 \mathcal R^{[L]}\colon X^L\times X^L\longrightarrow D_L^{\mathrm{ext}},
 \qquad
 \mathcal R^{[L]}(\mathbf x,\mathbf y)
 =\bigl(\tau_1(\mathbf x,\mathbf y),\ldots,\tau_d(\mathbf x,\mathbf y)\bigr),
\]
and let $R_{\balpha}^{[L]}=(\mathcal R^{[L]})^{-1}(\balpha)$.
This map is surjective: a label $\balpha$ is realized by choosing a pair
from each nonempty $R_i$ and placing it in $\alpha_i$ coordinates
$(0\le i\le d)$.
The triple
\[
 \frakX^{[L]}=\bigl(X^L,\mathcal R^{[L]},D_L^{\mathrm{ext}}\bigr)
\]
is called the \emph{extension} of $\frakX$ of length $L$.
\end{definition}

For $\boldsymbol\nu\in D_L^{\mathrm{ext}}$, put
\[
 \mathcal T_{\boldsymbol\nu}^{(L)}
 =\left\{(i_1,\ldots,i_L)\in\{0,\ldots,d\}^L:
   \bigl|\{\ell:i_\ell=s\}\bigr|=\nu_s\ (0\le s\le d)\right\}.
\]
The extension of a base association scheme $\frakX$ is a commutative association
scheme; its adjacency matrices and primitive idempotents are
\begin{align}
 A_{\balpha}^{[L]}
 &=\sum_{(i_1,\ldots,i_L)\in\mathcal T_{\balpha}^{(L)}}
   A_{i_1}\otimes\cdots\otimes A_{i_L},
 \label{eq:extension-adjacency-tensor}\\
 E_{\bbeta}^{[L]}
 &=\sum_{(j_1,\ldots,j_L)\in\mathcal T_{\bbeta}^{(L)}}
   E_{j_1}\otimes\cdots\otimes E_{j_L},
 \label{eq:extension-idempotent-tensor}
\end{align}
(cf.~\cite{BKZZ2025}).
With the notation $\binom{L}{r_0,\ldots,r_d}=L!/(r_0!\cdots r_d!)$ for
$r_0+\cdots+r_d=L$, the valencies and multiplicities are given by
\begin{align}
 \kappa_{\balpha}^{(L)}
 &=\binom{L}{L-|\balpha|,\alpha_1,\ldots,\alpha_d}
   \prod_{i=1}^dk_i^{\alpha_i},
\label{eq:extension-valency}\\
 \mu_{\bbeta}^{(L)}
 &=\binom{L}{L-|\bbeta|,\beta_1,\ldots,\beta_d}
   \prod_{j=1}^dm_j^{\beta_j}.
\label{eq:extension-multiplicity}
\end{align}

To express the eigenmatrices, let $p,q$ be nonnegative integers,
$\ba\in\NN^p$, $\bb\in\NN^q$, and $Z=(z_{ij})\in M_{p,q}(\CC)$.
For $C=(c_{ij})\in M_{p,q}(\NN)$, write
\[
 \rho_i(C)=\sum_{j=1}^qc_{ij},\qquad
 \sigma_j(C)=\sum_{i=1}^pc_{ij},\qquad
 |C|=\sum_{i=1}^p\sum_{j=1}^qc_{ij},
\]
and put $\boldsymbol\rho(C)=(\rho_1(C),\ldots,\rho_p(C))$.
When $p=0$ or $q=0$, the matrix set consists of a unique empty matrix.
Define the finite sum
\begin{equation}\label{eq:MT-hypergeom}
 \mathsf F_{p,q}(-\ba,-\bb;-L;Z)
 =
 \sum_{\substack{C\in M_{p,q}(\NN)\\|C|\le L}}
 \frac{\displaystyle
       \prod_{i=1}^p(-a_i)_{\rho_i(C)}
       \prod_{j=1}^q(-b_j)_{\sigma_j(C)}}
      {(-L)_{|C|}}
       \prod_{i=1}^p\prod_{j=1}^q
 \frac{z_{ij}^{c_{ij}}}{c_{ij}!}.
\end{equation}
Only matrices satisfying $\rho_i(C)\le a_i$ and $\sigma_j(C)\le b_j$
contribute, and $(-L)_{|C|}\ne0$ throughout the summation range.
This is the \emph{$(p+1,p+q+2)$-hypergeometric function} of Mizukawa--Tanaka
\cite{MizukawaTanaka2004}, with $n=p$ and $m-n-1=q$ in their notation;
they attribute its definition to K.~Aomoto and I.~M.~Gelfand.

For $0\le j,i\le d$, put
\[
 \omega^P_{ji}=\frac{P_i(j)}{k_i},\qquad
 \omega^Q_{ji}=\frac{Q_j(i)}{m_j}.
\]
These values are $1$ if $j=0$ or $i=0$, and \eqref{eq:PQ-entry-duality} gives
\begin{equation}\label{eq:omega-PQ}
 \omega^P_{ji}
 =\frac{P_i(j)}{k_i}
 =\frac{Q_j(i')}{m_j}
 =\omega^Q_{j,i'}.
\end{equation}
Define the following matrices, with rows indexed by the idempotent labels
$j$ and columns indexed by the relation labels $i$:
\begin{equation}\label{eq:Omega-PQ}
 \Omega^P=(1-\omega^P_{ji})_{1\le j,i\le d},
 \qquad
 \Omega^Q=(1-\omega^Q_{ji})_{1\le j,i\le d}.
\end{equation}
If $\frakX$ is symmetric, then $i'=i$, so $\Omega^P=\Omega^Q$.
In the convention of \eqref{eq:base-PQ}, the eigenmatrices of $\frakX^{[L]}$ are
\begin{align}
 P_{\balpha}^{[L]}(\bbeta)
 &=\kappa_{\balpha}^{(L)}
   \mathsf F_{d,d}(-\bbeta,-\balpha;-L;\Omega^P),
\label{eq:extension-P}\\
 Q_{\bbeta}^{[L]}(\balpha)
 &=\mu_{\bbeta}^{(L)}
   \mathsf F_{d,d}(-\bbeta,-\balpha;-L;\Omega^Q).
\label{eq:extension-Q}
\end{align}
The first formula is Theorem~1.1 and (6) in \cite{MizukawaTanaka2004},
rewritten in our notation; the second follows by applying the same formula
to the dual character algebra, as in Remark~1.3 in \cite{MizukawaTanaka2004}.
For $\bbeta,\balpha\in D_L^{\mathrm{ext}}$, the $\mathsf F_{d,d}$ functions
are called the \emph{multivariate Krawtchouk polynomials} \cite{Mizukawa2011}.

If $L=0$, the extension is the association scheme on a singleton, with
$A_{\bzero}^{[0]}=E_{\bzero}^{[0]}=(1)$ and
$\kappa_{\bzero}^{(0)}=\mu_{\bzero}^{(0)}
=P_{\bzero}^{[0]}(\bzero)=Q_{\bzero}^{[0]}(\bzero)=1$;
\eqref{eq:MT-hypergeom} has only the term $C=0$, equal to $1$.
If $d=0$, the base set $X$ is a singleton, $\NN^0=\{\bzero\}$, and
$\Omega^P,\Omega^Q$ are $0\times0$ matrices. Then $\mathsf F_{0,0}=1$
for every $L$, so all the formulas remain valid.

For $d\ge1$ and $L\ge1$, the extension is both $d$-variate $P$-polynomial
and $d$-variate $Q$-polynomial on $D_L^{\mathrm{ext}}$ with respect to
$\preceq_{\grlex}$ \cite{BKZZ2025}.
The following formulas for the Krein parameters, used in the proof of \cref{thm:Q},
also hold for $L=0$ or $d=0$.

\begin{lemma}\label[lemma]{lem:extension-Q-krein}
Let $L\ge0$.
For $\bbeta,\bgamma,\boldsymbol\delta\in D_L^{\mathrm{ext}}$ and the Krein parameters
$q_{\bbeta\bgamma}^{[L]\boldsymbol\delta}$ of $\frakX^{[L]}$, we have
\begin{equation}
  \label{eq:extension-Q-krein-triangularity}
  q_{\bbeta\bgamma}^{[L]\boldsymbol\delta}\neq 0
  \quad\text{implies}\quad
  \boldsymbol\delta\preceq_{\grlex}\bbeta+\bgamma.
\end{equation}
Moreover, if $\bbeta+\bgamma\in D_L^{\mathrm{ext}}$, that is,
$|\bbeta|+|\bgamma|\le L$, then explicitly
\begin{equation}\label{eq:extension-Q-top-krein}
 q_{\bbeta\bgamma}^{[L]\bbeta+\bgamma}
 =\prod_{j=1}^d\binom{\beta_j+\gamma_j}{\beta_j}>0,
\end{equation}
and this value is independent of $L$.
\end{lemma}

\begin{proof}
For $d,L\ge1$, \eqref{eq:extension-Q-krein-triangularity} follows from
\cref{prop:Q-criterion} \ref{item:q-criterion-krein-triangularity}, since the
extension is $d$-variate $Q$-polynomial with respect to $\preceq_{\grlex}$.
If $d=0$ or $L=0$, the extension is the association scheme on a singleton,
so both formulas are immediate.
To prove \eqref{eq:extension-Q-top-krein} for $d,L\ge1$, consider
$(v^LE_{\bbeta}^{[L]})\circ(v^LE_{\bgamma}^{[L]})$ using the tensor
expression \eqref{eq:extension-idempotent-tensor}.
Since $vE_0=J_X$ is the identity for the Hadamard product, if the positions
of the nontrivial factors $vE_j$ $(j\ge1)$ in the two tensors are disjoint,
their Hadamard product has $\beta_j+\gamma_j$ factors of type $j$, with
coefficient $1$.
If the positions overlap, expanding the products at the common positions
gives tensors with fewer than $|\bbeta|+|\bgamma|$ nontrivial factors.
Thus, when $|\bbeta|+|\bgamma|\le L$, only disjoint positions
contribute to $v^LE_{\bbeta+\bgamma}^{[L]}$.
For a fixed tensor with these factors, there are
$\binom{\beta_j+\gamma_j}{\beta_j}$ ways to choose which factors of type
$j$ come from the first tensor. Multiplying over $j$ gives
\eqref{eq:extension-Q-top-krein}.
\end{proof}

\section{Construction of generalized Johnson association schemes}\label{sec:construction}

Starting from a commutative association scheme $\frakX$, we construct
$\frakJ_{n,h}(\frakX)=\bigl(\Theta,\mathcal{S},\Lambda\bigr)$
on partial maps from $[n]$ to $X$, as described in \cref{sec:introduction}.
We compute products with the basic adjacency matrices and prove that
$\frakJ_{n,h}(\frakX)$ is a commutative association scheme and has the
multivariate $P$-polynomial property.
For the definition and criteria for the multivariate $P$-polynomial property,
see \cref{subsec:orders}.

\subsection{The construction theorem and the multivariate \texorpdfstring{$P$}{P}-polynomial structure}\label{subsec:main-results-P}

Define an order $\preceq_P$ on $\NN^{d+1}=\NN\times\NN^d$ by
\begin{equation}\label{eq:P-block-order}
 (t_1,\balpha_1)\prec_P(t_2,\balpha_2)
 \quad\Longleftrightarrow\quad
 \left\{
 \begin{array}{l}
 t_1<t_2,\quad\text{or}\\
 t_1=t_2\ \text{and}\ \balpha_1\prec_{\grlex}\balpha_2,
 \end{array}
 \right.
\end{equation}
and denote its restriction to $\Lambda$ by the same symbol.
This is a monomial order of $1$-block type, with the first coordinate as one
block and the remaining $d$ coordinates ordered by graded lexicographic order
for $d\ge1$. When $d=0$, it is the usual order on $\NN$.

\begin{theorem}\label{thm:scheme}
For any commutative association scheme
$\frakX=(X,\mathcal R,\{0,1,\ldots ,d\})$, any positive integer $n$,
and any $0\le h\le n$, the following hold.
\begin{enumerate}[label=\textup{(\roman*)},ref=\textup{(\roman*)}]
\item\label{item:scheme-association}
$\frakJ_{n,h}(\frakX)=\bigl(\Theta,\mathcal{S},\Lambda\bigr)$
is a commutative association scheme.
\item\label{item:scheme-P}
Suppose that $1\le h\le n-1$.
Then $\frakJ_{n,h}(\frakX)$ is $(d+1)$-variate $P$-polynomial on $\Lambda$
with respect to $\preceq_P$.
\end{enumerate}
\end{theorem}

\begin{remark}\label[remark]{rem:endpoint-cases}
When $h=0$, the set $\Theta$ is a singleton consisting only of the partial map
with empty domain (the empty map).
We have $\Lambda=\{(0,\bzero)\}$, and the only relation is
$R_{0,\bzero}=\Delta_{\Theta}=\Theta\times\Theta$,
so the construction in this case is the association scheme on a singleton.
When $h=n$, the elements of $\Theta$ are maps defined on all of $[n]$.
Thus $\Theta$ can be naturally identified with $X^n$, and $t$ is always $0$.
It follows that $\frakJ_{n,n}(\frakX)$ is isomorphic to $\frakX^{[n]}$.
\end{remark}

We now establish the basic properties needed to prove \cref{thm:scheme}
using \cref{prop:scheme-from-generators}.
For $(t,\balpha)\in \Lambda$, put $\alpha_0=h-t-|\balpha|$.

\begin{lemma}\label[lemma]{lem:nonempty-valency}
For every $(t,\balpha)\in \Lambda$, the relation $R_{t,\balpha}$ is nonempty.
In particular, $\mathcal{S}$ is surjective.
The valency of $R_{t,\balpha}$ is
\begin{equation}\label{eq:valency-generalized}
\begin{aligned}
 \kappa_{t,\balpha}
 &=\binom ht\binom{n-h}{t}v^t
 \binom{h-t}{\alpha_0,\alpha_1,\ldots,\alpha_d}
 \prod_{i=1}^{d}k_i^{\alpha_i}\\
 &=\binom ht\binom{n-h}{t}v^t\kappa_{\balpha}^{(h-t)}.
\end{aligned}
\end{equation}
\end{lemma}

\begin{proof}
For arbitrary $\theta,\eta\in\Theta$, the inequality
$2h-|\domn\theta\cap\domn\eta|\le n$ holds.
This gives $t(\theta,\eta)\le n-h$. Also, \eqref{eq:alpha-sum} gives
\[
 t(\theta,\eta)+\sum_{i=1}^d\alpha_i(\theta,\eta)\le h.
\]
Thus the map $\mathcal{S}$ in \cref{def:generalized-relation} is well defined.

Fix a partial map $\theta\in\Theta$, and count the partial maps
$\eta\in\Theta$ such that $(\theta,\eta)\in R_{t,\balpha}$.
Put $A=\domn\theta$ and $B=\domn\eta$.
First, there are $\binom ht\binom{n-h}t$ ways to form $B$ by removing $t$ points
from $A$ and adding $t$ points from $[n]\setminus A$.
The values of $\eta$ at points of $B\setminus A$ are arbitrary,
so there are $v^t$ ways to choose them.

The number of ways to partition the $h-t$ coordinates of $A\cap B$ into parts
of sizes $\alpha_0,\alpha_1,\ldots,\alpha_d$ is the multinomial coefficient
$\binom{h-t}{\alpha_0,\ldots,\alpha_d}$.
For each coordinate $a$ assigned type $i$, there are $k_i$ choices of $\eta(a)$
satisfying $(\theta(a),\eta(a))\in R_i$ for the fixed value $\theta(a)$.
Since $k_0=1$, multiplying these numbers gives \eqref{eq:valency-generalized}.
Every factor is positive, so the relation is nonempty.
\end{proof}

Let $i\mapsto i'$ be the correspondence in \ref{item:as-transpose} of
\cref{def:association-scheme}.
We have $0'=0$ and $(i')'=i$.
For $\balpha=(\alpha_1,\ldots,\alpha_d)$, define
\begin{equation}\label{eq:alpha-transpose}
 (\balpha')_i=\alpha_{i'}.
\end{equation}
Thus $\balpha'$ is obtained by permuting the components of $\balpha$
according to the transposition of relations $R_i^{\mathsf T}=R_{i'}$.
\begin{lemma}\label[lemma]{lem:partition-transpose}
We have $R_{0,\bzero}=\Delta_{\Theta}$ and
$R_{t,\balpha}^{\mathsf T}=R_{t,\balpha'}$ for $(t,\balpha)\in \Lambda$.
In particular, if $1\le h\le n-1$ and $\be_1,\ldots,\be_d$ are the
standard unit vectors of $\NN^d$, then
\[
 \{A_{(1,\bzero)},A_{(0,\be_1)},\ldots,A_{(0,\be_d)}\}
\]
is closed under transposition.
\end{lemma}
\begin{proof}
Having label $(0,\bzero)$ is equivalent to having the same domain and agreeing
in all values. Hence $R_{0,\bzero}=\Delta_{\Theta}$.
For $(\theta,\eta)\in R_{t,\balpha}$, reversing the pair to $(\eta,\theta)$
preserves the size of the intersection of the domains and replaces the relation
index $i$ in each common coordinate by $i'$.
This gives $R_{t,\balpha}^{\mathsf T}=R_{t,\balpha'}$.

Clearly $R_{1,\bzero}^{\mathsf T}=R_{1,\bzero}$.
Since $\be_i'=\be_{i'}$, we also have
$R_{0,\be_i}^{\mathsf T}=R_{0,\be_{i'}}$ for $1\le i\le d$.
As $0'=0$ and $i\mapsto i'$ is an involution, it permutes $\{1,\ldots,d\}$.
Hence the transpose of every displayed matrix belongs to the same set.
\end{proof}

In what follows, for $\balpha=(\alpha_1,\ldots,\alpha_d)$, put
\[
 \widealpha=(h-t-|\balpha|,\alpha_1,\ldots,\alpha_d)
             =(\alpha_0,\alpha_1,\ldots,\alpha_d)\in\NN^{d+1},
\]
and let $\bu_0,\bu_1,\ldots,\bu_d$ be the standard unit vectors of $\NN^{d+1}$.
Also put
\[
 \pr(a_0,a_1,\ldots,a_d)=(a_1,\ldots,a_d).
\]

The next lemma is needed to verify condition~\ref{item:generator-triangularity}
of \cref{prop:scheme-from-generators}.
The counting used in this lemma is essentially the same as that in
Theorem~4.5 in \cite{BKZZ2025}, but we give it again for the reader's convenience.
Note that we omit any term for which $\bgamma\in\mathbb Z^{d+1}$
has a negative component or $(u,\pr(\bgamma))\notin \Lambda$.
Also, $\bzero$ denotes the zero vector of the required dimension.

\begin{lemma}\label[lemma]{lem:generator-products}
The following hold.
\begin{enumerate}[label=\textup{(\roman*)},ref=\textup{(\roman*)}]
\item\label{item:C-product}
Let $h\ge1$, $1\le i\le d$, and $(t,\balpha)\in \Lambda$.
Then
\begin{equation}\label{eq:C-product}
 A_{0,\be_i}A_{t,\balpha}
 =t k_i A_{t,\balpha}
 +\sum_{j,\ell=0}^{d}
 (\alpha_\ell+1-\delta_{j\ell})p_{ij}^{\ell}
 A_{t,\pr(\widealpha-\bu_j+\bu_\ell)}.
\end{equation}
In particular, if $(t,\balpha+\be_i)\in \Lambda$, then the coefficient of
$A_{t,\balpha+\be_i}$ in $A_{0,\be_i}A_{t,\balpha}$ is $\alpha_i+1$.
Every other term $A_{t,\widetilde{\balpha}}$ satisfies $\widetilde{\balpha}\in\NN^d$ and
$|\widetilde{\balpha}|\le|\balpha|$.
\item\label{item:B-product}
Let $1\le h\le n-1$ and $(t,\balpha)\in \Lambda$.
Then
\begin{align}
 A_{1,\bzero}A_{t,\balpha}
 &=(t+1)^2\sum_{j=0}^{d}k_j
 A_{t+1,\pr(\widealpha-\bu_j)}
 \label{eq:B-product}\\
 &\quad+v t(n-h-t)A_{t,\balpha}\notag\\
 &\quad+t\sum_{j,\ell=0}^{d}
 (\alpha_\ell+1-\delta_{j\ell})k_j
 A_{t,\pr(\widealpha-\bu_j+\bu_\ell)}\notag\\
 &\quad+v(n-h-t+1)\sum_{\ell=0}^{d}(\alpha_\ell+1)
 A_{t-1,\pr(\widealpha+\bu_\ell)}.\notag
\end{align}
In particular, if $(t+1,\balpha)\in \Lambda$, then the coefficient of
$A_{t+1,\balpha}$ in $A_{1,\bzero}A_{t,\balpha}$ is $(t+1)^2$.
Every other term $A_{u,\widetilde{\balpha}}$ satisfies either
$u<t+1$ or $u=t+1$ and $|\widetilde{\balpha}|<|\balpha|$.
\end{enumerate}
\end{lemma}

\begin{proof}
We prove \ref{item:C-product}.
Fix two partial maps $\theta,\eta\in\Theta$.
The $(\theta,\eta)$ entry of $A_{0,\be_i}A_{t,\balpha}$ is the number of
partial maps $\xi\in\Theta$ satisfying
$(\theta,\xi)\in R_{0,\be_i}$ and $(\xi,\eta)\in R_{t,\balpha}$.
The maps $\theta$ and $\xi$ have the same domain, and
$(\theta(a),\xi(a))\in R_i$ at exactly one coordinate of that domain.

If this coordinate belongs to $\domn\theta\setminus\domn\eta$, there are
$t$ choices for it.
For the fixed value $\theta(a)$, there are $k_i$ choices of $\xi(a)$,
and the relation between $\theta$ and $\eta$ is unchanged.
Here $\balpha=(\alpha_1,\ldots,\alpha_d)\in\NN^d$ records the number of
common coordinates of $(\theta,\eta)$ belonging to each relation $R_i$.
This gives $t k_iA_{t,\balpha}$.

Next, suppose that the chosen coordinate $a$ belongs to $\domn\theta\cap\domn\eta$.
If $(\xi(a),\eta(a))\in R_j$ and $(\theta(a),\eta(a))\in R_\ell$, then
$\ell$ is the index of the relation containing the fixed pair
$(\theta(a),\eta(a))$.
By the definition of the intersection numbers,
\[
 \bigl|\{y\in X:(\theta(a),y)\in R_i,\ (y,\eta(a))\in R_j\}\bigr|
 =p_{ij}^{\ell}.
\]
Hence there are $p_{ij}^{\ell}$ choices of $\xi(a)\in X$ satisfying these two conditions.
The coordinate of type $j$ between the intermediate point $\xi$ and $\eta$
becomes a coordinate of type $\ell$ between $\theta$ and $\eta$,
so the latter type counts are $\widealpha-\bu_j+\bu_\ell$.
Here type $0$ means a coordinate whose two values agree, since $R_0=\Delta_X$.
For a fixed pair $(\theta,\eta)$ with these counts, there are
$\alpha_\ell+1-\delta_{j\ell}$ choices of the coordinate $a$ of type $\ell$.
Summing over $j,\ell$ gives \eqref{eq:C-product}.

For $|\widetilde{\balpha}|=|\balpha|+1$, it is necessary that $j=0$ and $\ell\ne0$.
Since $A_iA_0=A_i$, we have $p_{i0}^{\ell}=\delta_{i\ell}$,
so only $\ell=i$ occurs in this case.
Its coefficient is $\alpha_i+1$.

Next we prove \ref{item:B-product}.
We count the $(\theta,\eta)$ entry of $A_{1,\bzero}A_{t,\balpha}$.
Since $(\theta,\xi)\in R_{1,\bzero}$, we can write
\[
 \domn\xi=(\domn\theta\setminus\{a\})\cup\{b\}
\]
for some $a\in \domn\theta$ and $b\in[n]\setminus \domn\theta$.
Moreover, $\theta$ and $\xi$ agree on $\domn\theta\cap \domn\xi$, and $\xi(b)$ is arbitrary.
Suppose that $(\xi,\eta)\in R_{t,\balpha}$, and put $u=h-|\domn\theta\cap \domn\eta|$.
Since one point is replaced, $u$ is one of $t+1,t,t-1$.

Consider the case $u=t+1$.
Then we must have $a\in \domn\theta\setminus \domn\eta$ and $b\in \domn\eta\setminus \domn\theta$.
There are $t+1$ choices each for $a$ and $b$.
If $(\xi(b),\eta(b))\in R_j$, there are $k_j$ choices of $\xi(b)$.
The coordinate $b$, which has type $j$ between the intermediate point and $\eta$,
is not in the common domain of $\theta$ and $\eta$, so the latter type counts
are $\widealpha-\bu_j$.
This gives the first line of \eqref{eq:B-product}.

There are two possibilities when $u=t$.
The first is $a\in \domn\theta\setminus \domn\eta$ and $b\in[n]\setminus(\domn\theta\cup \domn\eta)$.
In this case, $|\domn\theta\setminus \domn\eta|=u=t$, so there are $t$ choices for $a$.
Also, $|\domn\theta\cup \domn\eta|=h+t$, so there are $n-h-t$ choices for $b$ and $v$ choices
for $\xi(b)$.
The relation type is unchanged, giving the second line.
The second possibility is $a\in \domn\theta\cap \domn\eta$ and $b\in \domn\eta\setminus \domn\theta$.
Suppose that $(\xi(b),\eta(b))\in R_j$ and $(\theta(a),\eta(a))\in R_\ell$.
There are $t$ choices for $b$, $k_j$ choices for $\xi(b)$, and
$\alpha_\ell+1-\delta_{j\ell}$ choices for $a$.
The type counts become $\widealpha-\bu_j+\bu_\ell$, giving the third line.

If $u=t-1$, then we must have $a\in \domn\theta\cap \domn\eta$ and
$b\in[n]\setminus(\domn\theta\cup \domn\eta)$.
There are $n-h-t+1$ choices for $b$ and $v$ choices for $\xi(b)$.
The coordinate $a$, absent from the common domain of the intermediate point
and $\eta$, becomes a coordinate of type $\ell$ in the common domain of
$\theta$ and $\eta$.
There are $\alpha_\ell+1$ choices for it, and the type counts become
$\widealpha+\bu_\ell$. This gives the fourth line.

Finally, only the case $j=0$ in the first line gives
$(u,\widetilde{\balpha})=(t+1,\balpha)$, and its coefficient is
$(t+1)^2k_0=(t+1)^2$, since $k_0=1$.
For every other term in the first line, $u=t+1$ and
$|\widetilde{\balpha}|=|\balpha|-1$.
All terms in the remaining lines have $u\le t$.
\end{proof}

\begin{lemma}\label[lemma]{lem:commuting}
If $h\ge1$, then $A_{0,\be_1},\ldots,A_{0,\be_d}$ commute with each other.
If $1\le h\le n-1$, then $A_{1,\bzero},A_{0,\be_1},\ldots,A_{0,\be_d}$
commute with each other.
\end{lemma}

\begin{proof}
Let $1\le i,j\le d$.
Substituting $(t,\balpha)=(0,\be_j)$ into
\cref{lem:generator-products}\ref{item:C-product} and simplifying gives
\begin{equation}\label{eq:CiCj-product}
 A_{0,\be_i}A_{0,\be_j}
 =hp_{ij}^{0}I_{\Theta}
  +\sum_{\ell=1}^{d}p_{ij}^{\ell}A_{0,\be_\ell}
  +(1+\delta_{ij})A_{0,\be_i+\be_j}.
\end{equation}
Here we used $p_{i0}^{\ell}=\delta_{i\ell}$.
If $(0,\be_i+\be_j)\notin \Lambda$, the last term is omitted.
By commutativity of the base association scheme, $p_{ij}^{\ell}=p_{ji}^{\ell}$,
so the right-hand side of \eqref{eq:CiCj-product} is symmetric in $i,j$.
Hence $A_{0,\be_i}A_{0,\be_j}=A_{0,\be_j}A_{0,\be_i}$.

Suppose that $1\le h\le n-1$.
Substituting $(t,\balpha)=(1,\bzero)$ into
\cref{lem:generator-products}\ref{item:C-product}, and
$(t,\balpha)=(0,\be_i)$ into \cref{lem:generator-products}\ref{item:B-product},
gives, respectively,
\[
 A_{0,\be_i}A_{1,\bzero}=k_iA_{1,\bzero}+A_{1,\be_i},\qquad
 A_{1,\bzero}A_{0,\be_i}=k_iA_{1,\bzero}+A_{1,\be_i}.
\]
Here we used $k_0=1$.
If $(1,\be_i)\notin \Lambda$, the term $A_{1,\be_i}$ is omitted.
Thus $A_{0,\be_i}A_{1,\bzero}=A_{1,\bzero}A_{0,\be_i}$.
\end{proof}

\begin{proof}[Proof of \cref{thm:scheme}]
First suppose that $1\le h\le n-1$.
We prove \ref{item:scheme-association} and \ref{item:scheme-P} simultaneously
by verifying that $\frakJ_{n,h}(\frakX)=\bigl(\Theta,\mathcal{S},\Lambda\bigr)$
satisfies \cref{prop:scheme-from-generators}.

By \eqref{eq:D}, the set $\Lambda$ is a finite downward closed subset of
$\NN^{d+1}$ containing the standard unit vectors $(1,\bzero)$ and
$(0,\be_i)$ ($1\le i\le d$).
The surjectivity of $\mathcal{S}$ was proved in \cref{lem:nonempty-valency}.
Condition~\ref{item:generator-partition} of \cref{prop:scheme-from-generators}
follows from \cref{lem:partition-transpose}, and
condition~\ref{item:generator-commutativity} of \cref{prop:scheme-from-generators}
follows from \cref{lem:commuting}.
It remains to verify condition~\ref{item:generator-triangularity} of
\cref{prop:scheme-from-generators}.

By \cref{lem:generator-products}\ref{item:C-product},
\[
 A_{0,\be_i}A_{t,\balpha}
 \in\Span\{A_{u,\widetilde{\balpha}}:(u,\widetilde{\balpha})\in \Lambda,
                    (u,\widetilde{\balpha})\preceq_P(t,\balpha+\be_i)\},
\]
and, if $(t,\balpha+\be_i)\in \Lambda$, its coefficient is $\alpha_i+1\ne0$.
Also, by \cref{lem:generator-products}\ref{item:B-product},
\[
 A_{1,\bzero}A_{t,\balpha}
 \in\Span\{A_{u,\widetilde{\balpha}}:(u,\widetilde{\balpha})\in \Lambda,
                    (u,\widetilde{\balpha})\preceq_P(t+1,\balpha)\},
\]
and, if $(t+1,\balpha)\in \Lambda$, its coefficient is $(t+1)^2\ne0$.
\Cref{prop:scheme-from-generators} therefore shows that
$\frakJ_{n,h}(\frakX)$ is a commutative association scheme and is
$(d+1)$-variate $P$-polynomial on $\Lambda$ with respect to $\preceq_P$.
Thus \ref{item:scheme-association} and \ref{item:scheme-P} hold in this range.

For the remaining cases $h=0,n$, it suffices to prove \ref{item:scheme-association}.
By \cref{rem:endpoint-cases}, $\frakJ_{n,h}(\frakX)$ is a commutative
association scheme in both cases.
\end{proof}

\subsection{The Schurian property}\label{subsec:schurian}

An association scheme is \emph{Schurian} if its relations are the
\emph{orbitals} of a transitive permutation group on its vertex set, namely,
the orbits of the diagonal action on ordered pairs.
A pair $(G,H)$, where $G$ is a finite group and $H$ is a subgroup, is a finite
\emph{Gelfand pair} if the permutation representation of $G$ on $G/H$ is
multiplicity free. A commutative association scheme is Schurian if and only if
it is obtained from a finite Gelfand pair; see \cite{BannaiIto1984,CST2006}.

Ceccherini-Silberstein--Scarabotti--Tolli \cite{CST2006} constructed
$\frakJ_{n,h}(\frakX)$ for Schurian base association schemes.
The following proposition characterizes when $\frakJ_{n,h}(\frakX)$ is Schurian.
When $h=0$, the association scheme $\frakJ_{n,0}(\frakX)$ is the association scheme on a singleton
regardless of the base association scheme and is always Schurian, so we assume $h\ge1$ below.
\begin{proposition}\label[proposition]{prop:GJ-schurian}
Let $1\le h\le n$.
Then $\frakJ_{n,h}(\frakX)$ is Schurian if and only if the base association scheme $\frakX$ is Schurian.
\end{proposition}

\begin{proof}
It is shown in \cite{CST2006} that if $\frakX$ is Schurian, then
$\frakJ_{n,h}(\frakX)$ is Schurian, so we prove the converse here.

Suppose that $\frakJ_{n,h}(\frakX)$ is Schurian, and let $\Gamma$ be a
permutation group whose orbitals are its relations.
If $|X|=1$, then $\frakX$ is the association scheme on a singleton and is Schurian.
Henceforth suppose that $|X|\ge2$.
Since $t(\theta,\eta)=0$ is equivalent to equality of the two domains,
this equivalence relation is preserved by $\Gamma$.
Fix an $h$-element subset $A\subseteq[n]$, and put
\[
 \Theta_A=\{\theta\in\Theta:\domn\theta=A\}\cong X^A.
\]
Let $\Gamma_{\Theta_A}=\{g\in\Gamma:g\Theta_A=\Theta_A\}$
be the setwise stabilizer of $\Theta_A$ in $\Gamma$.
Since $\Gamma$ preserves equality of domains, an element $g\in\Gamma$
satisfying $g\theta=\theta'$ with some $\theta,\theta'\in\Theta_A$ must
stabilize $\Theta_A$ setwise. 
Let $(\theta,\eta)$ and $(\theta',\eta')$ belong to
$R_{0,\balpha}\cap(\Theta_A\times\Theta_A)$.
Since $R_{0,\balpha}$ is an orbital of $\Gamma$,
there exists $g\in\Gamma$ such that
$g\theta=\theta'$ and $g\eta=\eta'$.
Since $g\theta=\theta'$ and $\theta,\theta'\in\Theta_A$,
$g$ stabilizes $\Theta_A$ setwise.
Hence $g\in\Gamma_{\Theta_A}$.
Thus $\Gamma_{\Theta_A}$ acts transitively on
$R_{0,\balpha}\cap(\Theta_A\times\Theta_A)$.

Consider the graph on $\Theta_A$ in which two points are adjacent if their values
differ in exactly one coordinate.
Its adjacency relation is the restriction of $\bigcup_{i=1}^dR_{0,\be_i}$
to $\Theta_A\times \Theta_A$, so $\Gamma_{\Theta_A}$ preserves adjacency in this graph.
Fix $a\in A$ and $z\in X^{A\setminus\{a\}}$, and put
\[
 \mathcal{C}=\{\theta\in \Theta_A:\theta(b)=z(b)\text{ for all }b\in A\setminus\{a\}\}.
\]
The set $\mathcal{C}$ is a maximal clique, and it is the unique maximal clique
containing any two of its distinct points.
Indeed, those two points differ only in coordinate $a$, so a point adjacent
to both must agree with both in every other coordinate.

Under the identification $\mathcal{C}\to X$, $\theta\mapsto\theta(a)$,
the relation $R_{0,\be_i}\cap(\mathcal{C}\times\mathcal{C})$ coincides with
$R_i$ of $\mathfrak{X}$.
Fix $i\ge1$, and let $(\theta,\eta)$ and $(\theta',\eta')$
belong to $R_{0,\be_i}\cap(\mathcal C\times\mathcal C)$.
There exists $g\in\Gamma_{\Theta_A}$ such that
$g\theta=\theta'$ and $g\eta=\eta'$.
Then $g\mathcal C$ is a maximal clique containing
$\theta'$ and $\eta'$.
Since $\mathcal C$ is the unique maximal clique containing
these two points, $g\mathcal C=\mathcal C$.
Hence $K=\{g\in\Gamma_{\Theta_A}:g\mathcal{C}=\mathcal{C}\}$ acts transitively
on each $R_i$ $(i\ge1)$.
Since $R_1$ has positive valency and $K$ acts transitively on $R_1$,
$K$ also acts transitively on $\mathcal C$.
Therefore the orbitals of the induced action of $K$ on
$\mathcal C$ are precisely $R_0,\ldots,R_d$.
Thus $\frakX$ is Schurian.
\end{proof}

\begin{remark}\label[remark]{rem:gelfand-realization}
If $\frakX$ is obtained from a finite Gelfand pair $(G,H)$, then this
construction of $\frakJ_{n,h}(\frakX)$ agrees with the construction of generalized Johnson association schemes in terms of
finite Gelfand pairs by Ceccherini-Silberstein--Scarabotti--Tolli \cite{CST2006}.
According to \cite{CST2006}, $\frakJ_{n,h}(\frakX)$ corresponds to the finite
Gelfand pair $(G\wr S_n,(H\wr S_{h})\times (G\wr S_{n-h}))$.
\end{remark}

\section{\texorpdfstring{Primitive idempotents and the multivariate\\
$Q$-polynomial property of $\frakJ_{n,h}$}{Primitive idempotents and the multivariate Q-polynomial property of J(n,h)}}\label{sec:primitive-idempotents_of_GJS}

In this section, we construct the primitive idempotents of $\frakJ_{n,h}(\frakX)$
and determine its two eigenmatrices. We then use their product expressions
to prove the multivariate $Q$-polynomial property.

\subsection{Construction of the primitive idempotents}\label{subsec:primitive-idempotents_of_GJS}

We determine the primitive idempotents and the two eigenmatrices of the
generalized Johnson association scheme $\frakJ_{n,h}(\frakX)=\bigl(\Theta,\mathcal{S},\Lambda\bigr)$.
Put
\[
 \Lambda^\ast=\{(\bbeta,k)\in\NN^d\times\NN:
       k\le n-h,\ |\bbeta|+k\le h\}
\]
for the index set of the primitive idempotents.
The map $(\bbeta,k)\mapsto(k,\bbeta)$ is a bijection from $\Lambda^\ast$ to $\Lambda$.

Let $\{E_j\}_{j=0}^d$ be the primitive idempotents of the base
association scheme $\frakX$, with $E_0=v^{-1}J_X$,
and put $U_j=\im E_j\subseteq\CC^X$.
Then
\[
\CC^X=\bigoplus_{j=0}^d U_j,
\qquad \dim U_j=m_j\quad(0\le j\le d).
\]
Putting $u_0=v^{-1/2}\one$, we have $U_0=\CC u_0$.

Let $\star\notin X$ be a symbol, and let $e_\star$ be a unit vector orthogonal to $\CC^X$.
Equip
\[
 W=\CC e_\star\oplus \CC^X
\]
with the orthogonal direct sum inner product.
Thus $\{e_\star\}\cup\{e_y:y\in X\}$ is the standard orthonormal basis of $W$.
For $B\in\binom{[n]}{h}$ and $a\in[n]$, put
\[
 \cH_a(B)=
 \begin{cases}
  \CC^X,&a\in B,\\
  \CC e_\star,&a\notin B
 \end{cases},
\]
and define the subspace of $W^{\otimes n}$ by
\begin{equation}\label{eq:tensor-weight-space}
 \cH_h=
 \mathop{\bigoplus}_{B\in\binom{[n]}{h}}
 \bigotimes_{a=1}^n\cH_a(B).
\end{equation}
For $\theta\in\Theta$, identify $e_\theta\in\CC^\Theta$ with
\[
 \bigotimes_{a=1}^n f_a(\theta),
 \qquad
 f_a(\theta)=
 \begin{cases}
  e_{\theta(a)},&a\in\domn\theta,\\
  e_\star,&a\notin\domn\theta.
 \end{cases}
\]
This identifies $\CC^\Theta$ with $\cH_h$.

Let $(\bbeta,k)\in \Lambda^\ast$.
By the definition of $\Lambda^\ast$, we have
$0\le h-|\bbeta|\le n-|\bbeta|$ and $0\le k\le\min\{h-|\bbeta|,n-h\}$.
Choose pairwise disjoint subsets $B^{(j)}\subseteq[n]$, $|B^{(j)}|=\beta_j$ $(1\le j\le d)$.
Place $U_j$ in each factor indexed by $B^{(j)}$.
In the remaining $n-|\bbeta|$ factors, place either $\CC e_\star$ or $U_0$,
with $h-|\bbeta|$ factors equal to $U_0$.
The direct sum of these spaces, with $U_0$ in $h-|\bbeta|$ factors and
$\CC e_\star$ in $n-h$ factors, can be identified with
$\CC^{\binom{[n-|\bbeta|]}{h-|\bbeta|}}$.
Using $\im E_k^{J(n-|\bbeta|,h-|\bbeta|)}$ from \cref{subsec:Johnson},
take the subspace corresponding to fixed $B^{(1)},\ldots,B^{(d)}$ to be
\[
 \bigl(\im E_k^{J(n-|\bbeta|,h-|\bbeta|)}\bigr)\otimes
 \bigotimes_{j=1}^dU_j^{\otimes B^{(j)}}.
\]
Denote the orthogonal direct sum of these spaces over all pairwise disjoint
$B^{(1)},\ldots,B^{(d)}$ by $\mathsf W_{\bbeta,k}$.

First we verify that
\begin{equation}\label{eq:orthogonal-decomposition}
 \CC^{\Theta}
 =\bigoplus_{(\bbeta,k)\in \Lambda^\ast}\mathsf W_{\bbeta,k}.
\end{equation}
Applying the orthogonal decomposition
\[
 W=\CC e_\star\oplus 
   \left(\bigoplus_{j=0}^dU_j\right)
\]
to each factor in \eqref{eq:tensor-weight-space} decomposes $\cH_h$ into
an orthogonal direct sum of subspaces with either $\CC e_\star$ or $U_j$
in each factor.
Once the positions of $U_j$ $(1\le j\le d)$ are fixed as $B^{(j)}$,
the remaining $n-|\bbeta|$ factors contain $h-|\bbeta|$ copies of $U_0$.
Identify this part with $\CC^{\binom{[n-|\bbeta|]}{h-|\bbeta|}}$
and decompose it orthogonally into the images of the primitive idempotents
of the Johnson association scheme.
The condition $0\le k\le\min\{h-|\bbeta|,n-h\}$ is precisely the condition
$(\bbeta,k)\in \Lambda^\ast$, giving \eqref{eq:orthogonal-decomposition}.
Also, counting the choices of $B^{(1)},\ldots,B^{(d)}$ and taking the dimensions of
the factors gives
\begin{equation}\label{eq:globalmultiplicity}
 \dim\mathsf W_{\bbeta,k}
 =\binom{n}{n-|\bbeta|,\beta_1,\ldots,\beta_d}
  d_{k}^{n-|\bbeta|}\prod_{j=1}^dm_j^{\beta_j}
 =\binom{n}{|\bbeta|} d_{k}^{n-|\bbeta|}\,
  \frac{|\bbeta|!}{\beta_1!\cdots\beta_d!}
  \prod_{j=1}^dm_j^{\beta_j}.
\end{equation}
All these dimensions are positive.
Henceforth denote the value in \eqref{eq:globalmultiplicity} by $M_{\bbeta,k}$.
Let $E_{\bbeta,k}$ be the orthogonal projection onto $\mathsf W_{\bbeta,k}$.

\subsection{Primitive idempotents and product expressions for the eigenmatrices}\label{subsec:hypergeom}

We show that the orthogonal projections $E_{\bbeta,k}$ constructed above are
the primitive idempotents of $\frakJ_{n,h}(\frakX)$ and compute the corresponding eigenmatrices.

\begin{theorem}\label{thm:factorization-main}
The primitive idempotents of $\frakJ_{n,h}(\frakX)$ are
$\{E_{\bbeta,k}:(\bbeta,k)\in \Lambda^\ast\}$.
For $(t,\balpha)\in \Lambda$ and $(\bbeta,k)\in \Lambda^\ast$,
the first and second eigenmatrices are given, respectively, by
\begin{equation}\label{eq:P-factorization}
 P_{t,\balpha}(\bbeta,k)
 =
 \begin{cases}
 \displaystyle
 v^t\,\mathsf E_t(n-|\bbeta|,h-|\bbeta|;k)\,
 \kappa_{\balpha}^{(h-t)}
 \mathsf F_{d,d}(-\bbeta,-\balpha;-(h-t);\Omega^P),
 &|\bbeta|\le h-t,\\[2mm]
 0,&|\bbeta|>h-t,
 \end{cases}
\end{equation}
and
\begin{equation}\label{eq:Q-factorization}
 Q_{\bbeta,k}(t,\balpha)
 =
 \begin{cases}
 \displaystyle
 \frac{\binom{n}{|\bbeta|}}{\binom{h}{|\bbeta|}}\,
 \mathsf H_k(n-|\bbeta|,h-|\bbeta|;t)\,
 \mu_{\bbeta}^{(h-t)}
 \mathsf F_{d,d}(-\bbeta,-\balpha;-(h-t);\Omega^Q),
 &|\bbeta|\le h-t,\\[2mm]
 0,&|\bbeta|>h-t.
 \end{cases}
\end{equation}

\end{theorem}

\begin{proof}
We compute the entry
$(E_{\bbeta,k})_{\theta,\eta}=\langle E_{\bbeta,k}e_\theta,E_{\bbeta,k}e_\eta\rangle_\Theta$
of the orthogonal projection $E_{\bbeta,k}$ onto $\mathsf W_{\bbeta,k}$
at $(\theta,\eta)\in R_{t,\balpha}$.
We have $|\domn\theta\cap\domn\eta|=h-t$, and put $\alpha_0=h-t-|\balpha|$.
For each $a\in \domn\theta\cap\domn\eta$, choose $i_a$ such that
$(\theta(a),\eta(a))\in R_{i_a}$.
A contribution from fixed $B^{(1)},\ldots,B^{(d)}$ requires
$B^{(1)}\sqcup\cdots\sqcup B^{(d)}\subseteq \domn\theta\cap\domn\eta$.
Indeed, if $a\in B^{(j)}$ does not belong to $\domn\theta$, then the $a$th
factor of the standard basis tensor $e_\theta$ is $e_\star$, which is
orthogonal to $U_j$. Hence the projection of $e_\theta$ onto the summand
corresponding to these fixed sets is zero. The same argument applies to
$e_\eta$ if $a\notin\domn\eta$. In either case, the contribution of this
summand to the projection entry is zero.
Thus, if $|\bbeta|>h-t$, then $(E_{\bbeta,k})_{\theta,\eta}=0$.

Henceforth suppose that $|\bbeta|\le h-t$.
When $B^{(1)}\sqcup\cdots\sqcup B^{(d)}\subseteq \domn\theta\cap\domn\eta$, the product of the projection
entries on the $B^{(j)}$ is, by \eqref{eq:base-PQ},
\[
 \prod_{j=1}^d\prod_{a\in B^{(j)}}(E_j)_{\theta(a),\eta(a)}
 =v^{-|\bbeta|}\prod_{j=1}^d\prod_{a\in B^{(j)}}Q_j(i_a).
\]
In the remaining $n-|\bbeta|$ factors, both sets
\[
\begin{aligned}
 T_\theta&=\domn\theta\setminus\bigl(B^{(1)}\sqcup\cdots\sqcup B^{(d)}\bigr),\\
 T_\eta&=\domn\eta\setminus\bigl(B^{(1)}\sqcup\cdots\sqcup B^{(d)}\bigr)
\end{aligned}
\]
have $h-|\bbeta|$ elements, and the Johnson distance between $T_\theta$ and $T_\eta$ is
\[
 h-|\bbeta|-|T_\theta\cap T_\eta|=h-|\bbeta|-(h-t-|\bbeta|)=t.
\]
To realize the Johnson association scheme in these tensor factors, put
$C=[n]\setminus\bigcup_{j=1}^d B^{(j)}$ and $H=h-|\bbeta|$.
For $T\in\binom{C}{H}$, let $b_T$ be the tensor with $u_0$ in the positions
indexed by $T$ and $e_\star$ in those indexed by $C\setminus T$.
The tensors $b_T$ form an orthonormal family, and the isometry $e_T\mapsto b_T$
identifies $\CC^{\binom{C}{H}}$ with their span.
We transfer the projection $E_k^{J(n-|\bbeta|,H)}$ to this span by this isometry.
Thus, by \eqref{eq:johnson-projector}, the $(T_\theta,T_\eta)$ entry of
$E_k^{J(n-|\bbeta|,h-|\bbeta|)}$ is
$\binom{n-|\bbeta|}{h-|\bbeta|}^{-1}\mathsf H_k(n-|\bbeta|,h-|\bbeta|;t)$.
Here $t\le h-|\bbeta|$ and $t\le n-h$, so the Hahn polynomial is evaluated
within its domain.
Since $\langle u_0,e_y\rangle_X=v^{-1/2}$, a factor
$v^{-(h-|\bbeta|)/2}$ occurs on each of the $\theta$ and $\eta$ sides.
Hence the contribution from fixed $B^{(1)},\ldots,B^{(d)}$ is
\[
 \frac{v^{-h}}{\binom{n-|\bbeta|}{h-|\bbeta|}}\mathsf H_k(n-|\bbeta|,h-|\bbeta|;t)
 \prod_{j=1}^d\prod_{a\in B^{(j)}}Q_j(i_a).
\]

Sum over all $B^{(1)},\ldots,B^{(d)}$.
Write $[z_1^{\beta_1}\cdots z_d^{\beta_d}]f$ for the coefficient of
$z_1^{\beta_1}\cdots z_d^{\beta_d}$ in a polynomial $f$ in the variables $z_1,\ldots,z_d$.
By $Q_j(0)=m_j$ and the coefficient formulas (1) and (2) in
\cite{Mizukawa2011}, we obtain
\begin{equation}\label{eq:extension-Q-coefficient}
\begin{aligned}
 &\sum_{\substack{B^{(1)}\sqcup\cdots\sqcup B^{(d)}\subseteq \domn\theta\cap\domn\eta\\
                   |B^{(j)}|=\beta_j\ (1\le j\le d)}}
    \prod_{j=1}^d\prod_{a\in B^{(j)}}Q_j(i_a)\\
 &\qquad=[z_1^{\beta_1}\cdots z_d^{\beta_d}]
   \prod_{i=0}^d\left(1+\sum_{j=1}^dz_jQ_j(i)\right)^{\alpha_i}\\
 &\qquad=\binom{h-t}{h-t-|\bbeta|,\beta_1,\ldots,\beta_d}
   \prod_{j=1}^dm_j^{\beta_j}\,
   \mathsf F_{d,d}(-\bbeta,-\balpha;-(h-t);\Omega^Q)\\
 &\qquad=\mu_{\bbeta}^{(h-t)}
   \mathsf F_{d,d}(-\bbeta,-\balpha;-(h-t);\Omega^Q).
\end{aligned}
\end{equation}
The last equality uses \eqref{eq:extension-multiplicity}.
Thus we have
\begin{equation}\label{eq:projector-kernel}
 (E_{\bbeta,k})_{\theta,\eta}
 =\begin{cases}
 \displaystyle
 \frac{v^{-h}}{\binom{n-|\bbeta|}{h-|\bbeta|}}\mathsf H_k(n-|\bbeta|,h-|\bbeta|;t)
 \mu_{\bbeta}^{(h-t)}
 \mathsf F_{d,d}(-\bbeta,-\balpha;-(h-t);\Omega^Q),&|\bbeta|\le h-t,\\[2mm]
 0,&|\bbeta|>h-t.
 \end{cases}
\end{equation}
Since this value depends only on $(t,\balpha)$, each $E_{\bbeta,k}$ belongs
to the Bose--Mesner algebra of $\frakJ_{n,h}(\frakX)$.
By \eqref{eq:orthogonal-decomposition}, these are mutually orthogonal nonzero
projections whose sum is $I_{\Theta}$.
Their number $|\Lambda^\ast|=|\Lambda|$ equals the dimension of the algebra,
so these projections form a basis of the algebra.
Hence each $\cA(\frakJ_{n,h}(\frakX))E_{\bbeta,k}$ is one-dimensional,
and $E_{\bbeta,k}$ is a primitive idempotent.
This gives a complete system of primitive idempotents and their multiplicities
\eqref{eq:globalmultiplicity}.

Next we determine the second eigenmatrix.
Let $Q_{\bbeta,k}(t,\balpha)$ be the value taken by
$|\Theta|E_{\bbeta,k}$ on $R_{t,\balpha}$.
From $|\Theta|=\binom nhv^h$ and
$\binom nh/\binom{n-|\bbeta|}{h-|\bbeta|}
=\binom{n}{|\bbeta|}/\binom{h}{|\bbeta|}$,
when $|\bbeta|\le h-t$, \eqref{eq:projector-kernel} gives
\begin{equation}\label{eq:second-eigenmatrix}
 Q_{\bbeta,k}(t,\balpha)
 =\frac{\binom{n}{|\bbeta|}}{\binom{h}{|\bbeta|}}\,
  \mathsf H_k(n-|\bbeta|,h-|\bbeta|;t)\mu_{\bbeta}^{(h-t)}
  \mathsf F_{d,d}(-\bbeta,-\balpha;-(h-t);\Omega^Q),
\end{equation}
which is the first line of \eqref{eq:Q-factorization}.
When $|\bbeta|>h-t$, \eqref{eq:projector-kernel} gives
$Q_{\bbeta,k}(t,\balpha)=0$, proving the second line as well.

Finally, we determine the first eigenmatrix.
Let $P_{t,\balpha}(\bbeta,k)$ be the eigenvalue of $A_{t,\balpha}$ on
$E_{\bbeta,k}$, and put $\balpha'=(\alpha_{1'},\ldots,\alpha_{d'})$.
Applying \eqref{eq:PQ-entry-duality} gives
\begin{equation}\label{eq:global-PQ-duality}
 \frac{P_{t,\balpha}(\bbeta,k)}{\kappa_{t,\balpha}}
 =\frac{Q_{\bbeta,k}(t,\balpha')}{M_{\bbeta,k}}.
\end{equation}
When $|\bbeta|>h-t$, the right-hand side vanishes, so the first eigenvalue
also vanishes.
Henceforth suppose that $|\bbeta|\le h-t$.
By \eqref{eq:omega-PQ} and \eqref{eq:Omega-PQ}, relabeling the columns by
$i\mapsto i'$ in the sum \eqref{eq:MT-hypergeom} gives
\[
 \mathsf F_{d,d}(-\bbeta,-\balpha';-(h-t);\Omega^Q)
 =\mathsf F_{d,d}(-\bbeta,-\balpha;-(h-t);\Omega^P).
\]
Also, \eqref{eq:globalmultiplicity} and \eqref{eq:extension-multiplicity} give
\[
 M_{\bbeta,k}
 =\frac{\binom{n}{|\bbeta|}}{\binom{h}{|\bbeta|}}d_{k}^{n-|\bbeta|}\mu_{\bbeta}^{(h)},
 \qquad
 \frac{\mu_{\bbeta}^{(h-t)}}{\mu_{\bbeta}^{(h)}}
 =\frac{\binom{h-t}{|\bbeta|}}{\binom{h}{|\bbeta|}}.
\]
Substituting these, \eqref{eq:valency-generalized}, and \eqref{eq:second-eigenmatrix} into
\eqref{eq:global-PQ-duality} gives
\[
\begin{aligned}
 P_{t,\balpha}(\bbeta,k)
 &=v^t\binom ht\binom{n-h}{t}
   \frac{\binom{h-t}{|\bbeta|}}{\binom{h}{|\bbeta|}}
   \frac{\mathsf H_k(n-|\bbeta|,h-|\bbeta|;t)}{d_{k}^{n-|\bbeta|}}
   \kappa_{\balpha}^{(h-t)}
   \mathsf F_{d,d}(-\bbeta,-\balpha;-(h-t);\Omega^P)\\
 &=v^t\binom{h-|\bbeta|}{t}\binom{n-h}{t}
   \frac{\mathsf H_k(n-|\bbeta|,h-|\bbeta|;t)}{d_{k}^{n-|\bbeta|}}
   \kappa_{\balpha}^{(h-t)}
   \mathsf F_{d,d}(-\bbeta,-\balpha;-(h-t);\Omega^P)\\
 &=v^t\mathsf E_t(n-|\bbeta|,h-|\bbeta|;k)\kappa_{\balpha}^{(h-t)}
   \mathsf F_{d,d}(-\bbeta,-\balpha;-(h-t);\Omega^P).
\end{aligned}
\]
The second equality uses
$\binom ht\binom{h-t}{|\bbeta|}/\binom{h}{|\bbeta|}=\binom{h-|\bbeta|}{t}$,
and the third uses \eqref{eq:Johnson-PQ-relation}.
This proves \eqref{eq:P-factorization}.
\end{proof}

\begin{remark}\label[remark]{rem:factorization-meaning}
The factor
$\kappa_{\balpha}^{(h-t)}\mathsf F_{d,d}(-\bbeta,-\balpha;-(h-t);\Omega^P)$
in \eqref{eq:P-factorization} is a first eigenvalue of the extension
$\frakX^{[h-t]}$ of the base association scheme, and the factor
$\mu_{\bbeta}^{(h-t)}\mathsf F_{d,d}(-\bbeta,-\balpha;-(h-t);\Omega^Q)$
in \eqref{eq:Q-factorization} is a second eigenvalue of the same association scheme.
When $h=t$, this extension is the association scheme on a singleton.
Thus the two formulas are products of an eigenvalue of the Johnson association scheme
and an eigenvalue of $\frakX^{[h-t]}$, multiplied by the displayed elementary factors.
The factor $v^t$ arises from the number of ways to choose values in $X$ at the
$t$ coordinates of $\domn\eta\setminus\domn\theta$ when the first partial
map $\theta$ is fixed.
\end{remark}

\subsection{The multivariate \texorpdfstring{$Q$}{Q}-polynomial property of \texorpdfstring{$\frakJ_{n,h}(\frakX)$}{J(n,h)(X)}}\label{subsec:main-results-Q}

In this subsection, we prove the multivariate $Q$-polynomial property of $\frakJ_{n,h}(\frakX)$ from the
product expressions for the eigenmatrices and orthogonality.
To this end, we verify the triangularity condition for the Krein parameters
and the nonvanishing of the leading coefficients.
Throughout this subsection, suppose that $1\le h\le n-1$.

For the primitive idempotents, we again use $\Lambda^\ast$.
Define an order $\preceq_Q$ on $\NN^d\times\NN$ by
\begin{equation}\label{eq:Q-block-order}
 (\bbeta_1,k_1)\prec_Q(\bbeta_2,k_2)
 \quad\Longleftrightarrow\quad
 \left\{
 \begin{array}{l}
 \bbeta_1\prec_{\grlex}\bbeta_2,\quad\text{or}\\
 \bbeta_1=\bbeta_2\ \text{and}\ k_1<k_2
 \end{array}
 \right.,
\end{equation}
and denote its restriction to $\Lambda^\ast$ by the same symbol.
For $d\ge1$, this is a monomial order of $d$-block type, with the first $d$
coordinates forming one block, and hence it is of $d$-elimination type.
For $d=0$, take the usual order on $\NN$.

\begin{theorem}\label{thm:Q}
Suppose that $1\le h\le n-1$.
Then $\frakJ_{n,h}(\frakX)$ is $(d+1)$-variate $Q$-polynomial on
$\Lambda^\ast$ with respect to $\preceq_Q$.
The primitive idempotents corresponding to the $E_{\eps_j}$ in
\cref{def:multivariate-Q-polynomial} can be taken to be
$E_{\be_j,0}$ $(1\le j\le d)$ and $E_{\bzero,1}$.
\end{theorem}

\begin{proof}
The set $\Lambda^\ast$ is downward closed and contains $(\bzero,1)$ and
$(\be_j,0)$ $(1\le j\le d)$.
Also, $E_{\bzero,0}=|\Theta|^{-1}J_{\Theta}$.
It therefore suffices to verify the triangularity and nonvanishing conditions
in \cref{prop:Q-criterion} \ref{item:q-criterion-krein-triangularity}.

Take arbitrary $(\bbeta,k),(\bgamma,\ell),(\boldsymbol\delta,r)\in \Lambda^\ast$,
and, for $\boldsymbol\rho\in\{\bbeta,\bgamma,\boldsymbol\delta\}$, put
\[
 C_{\boldsymbol\rho}
 =\frac{\binom n{|\boldsymbol\rho|}}{\binom h{|\boldsymbol\rho|}}>0.
\]
By \eqref{eq:Q-factorization} and \eqref{eq:extension-Q}, for
$(\boldsymbol\rho,a)\in \Lambda^\ast$ we have
\begin{equation}\label{eq:Q-extension-factorization}
 Q_{\boldsymbol\rho,a}(t,\balpha)
 =C_{\boldsymbol\rho}\,
  \mathsf H_a(n-|\boldsymbol\rho|,h-|\boldsymbol\rho|;t)
  Q_{\boldsymbol\rho}^{[h-t]}(\balpha)
 \qquad(|\boldsymbol\rho|\le h-t),
\end{equation}
and the left-hand side is zero if $|\boldsymbol\rho|>h-t$.
Moreover, \eqref{eq:valency-generalized} and \eqref{eq:globalmultiplicity} give
\[
 \kappa_{t,\balpha}
 =v^h\binom ht\binom{n-h}t
   \frac{\kappa_{\balpha}^{(h-t)}}{v^{h-t}},
 \qquad
 M_{\boldsymbol\delta,r}
 =C_{\boldsymbol\delta}d_{r}^{n-|\boldsymbol\delta|}\mu_{\boldsymbol\delta}^{(h)}.
\]

First, we express the Krein parameters in terms of the Krein parameters of
an extension and Hahn polynomials.
Applying \eqref{eq:Krein-eigenmatrix-formula2} to $\frakX^{[h-t]}$, we obtain,
for $|\bbeta|,|\bgamma|,|\boldsymbol\delta|\le h-t$,
\[
 \sum_{\balpha\in D_{h-t}^{\mathrm{ext}}}
 \frac{\kappa_{\balpha}^{(h-t)}}{v^{h-t}}
 \overline{Q_{\bbeta}^{[h-t]}(\balpha)
           Q_{\bgamma}^{[h-t]}(\balpha)}
 Q_{\boldsymbol\delta}^{[h-t]}(\balpha)
 =\mu_{\boldsymbol\delta}^{(h-t)}
   q_{\bbeta\bgamma}^{[h-t]\boldsymbol\delta}.
\]
Furthermore, by \eqref{eq:extension-multiplicity},
\[
 \binom ht\binom{n-h}t\mu_{\boldsymbol\delta}^{(h-t)}
 =\mu_{\boldsymbol\delta}^{(h)}
   \binom{h-|\boldsymbol\delta|}t\binom{n-h}t
 =\mu_{\boldsymbol\delta}^{(h)}\nu_t^{n-|\boldsymbol\delta|,h-|\boldsymbol\delta|}.
\]
Hence, applying \eqref{eq:Krein-eigenmatrix-formula2} to $\frakJ_{n,h}(\frakX)$
and substituting \eqref{eq:Q-extension-factorization}, we obtain
\begin{equation}\label{eq:Q-krein-reduction}
\begin{aligned}
 q_{(\bbeta,k)(\bgamma,\ell)}^{(\boldsymbol\delta,r)}
 &=
 \frac{C_{\bbeta}C_{\bgamma}}{\binom nh\,d_{r}^{n-|\boldsymbol\delta|}}
 \sum_{t=0}^{U}\nu_t^{n-|\boldsymbol\delta|,h-|\boldsymbol\delta|}\,
 q_{\bbeta\bgamma}^{[h-t]\boldsymbol\delta}\\
 &\qquad{}\times
 \mathsf H_k(n-|\bbeta|,h-|\bbeta|;t)
 \mathsf H_\ell(n-|\bgamma|,h-|\bgamma|;t)
 \mathsf H_r(n-|\boldsymbol\delta|,h-|\boldsymbol\delta|;t),
\end{aligned}
\end{equation}
where $U=\min\bigl\{h-\max\{|\bbeta|,|\bgamma|,|\boldsymbol\delta|\},n-h\bigr\}$.
Here we used the fact that the Hahn polynomials are real.

By \eqref{eq:extension-Q-krein-triangularity}, if
$\boldsymbol\delta\succ_{\grlex}\bbeta+\bgamma$, then every term in
\eqref{eq:Q-krein-reduction} is zero.
If $\boldsymbol\delta\prec_{\grlex}\bbeta+\bgamma$, then
$(\boldsymbol\delta,r)\prec_Q(\bbeta+\bgamma,k+\ell)$ regardless of the value of $r$.
Thus the only remaining case for the triangularity condition is
$\boldsymbol\delta=\bbeta+\bgamma$.

Henceforth suppose that $\boldsymbol\delta=\bbeta+\bgamma$. Then
$U=\min\{h-|\bbeta|-|\bgamma|,n-h\}$.
Since $(\boldsymbol\delta,r)\in \Lambda^\ast$, we have
$h-|\bbeta|-|\bgamma|\ge0$ and $0\le r\le U$.
By \eqref{eq:extension-Q-top-krein}, for $0\le t\le U$ we have
\[
 q_{\bbeta\bgamma}^{[h-t]\bbeta+\bgamma}
 =c_{\bbeta,\bgamma}
 :=\prod_{j=1}^d\binom{\beta_j+\gamma_j}{\beta_j}>0,
\]
and this value is independent of $t$.
Thus, defining the polynomial $p(t)$ in $t$ by
\[
 p(t)=
 \mathsf H_k(n-|\bbeta|,h-|\bbeta|;t)
 \mathsf H_\ell(n-|\bgamma|,h-|\bgamma|;t),
\]
we obtain
\begin{equation}\label{eq:Q-krein-top-block}
 q_{(\bbeta,k)(\bgamma,\ell)}^{(\bbeta+\bgamma,r)}
 =
 \frac{C_{\bbeta}C_{\bgamma}c_{\bbeta,\bgamma}}
      {\binom nh\,d_{r}^{n-|\bbeta|-|\bgamma|}}
 \sum_{t=0}^{U}\nu_t^{n-|\bbeta|-|\bgamma|,h-|\bbeta|-|\bgamma|}
 p(t)\mathsf H_r(n-|\bbeta|-|\bgamma|,h-|\bbeta|-|\bgamma|;t).
\end{equation}

The polynomial $p$ has degree $k+\ell$.
Applying \cref{lem:Johnson-quotient-degree} with
$N=n-|\bbeta|-|\bgamma|$, $H=h-|\bbeta|-|\bgamma|$, and $T=U$
to the sum in \eqref{eq:Q-krein-top-block}, we obtain
\[
 q_{(\bbeta,k)(\bgamma,\ell)}^{(\bbeta+\bgamma,r)}\ne0
 \quad\Longrightarrow\quad r\le k+\ell.
\]
This establishes the triangularity condition in all cases.

Finally, suppose that $(\bbeta+\bgamma,k+\ell)\in\Lambda^\ast$.
Then $k+\ell\le U$. Since $\deg p=k+\ell$, the second part of
\cref{lem:Johnson-quotient-degree} shows that the sum in
\eqref{eq:Q-krein-top-block} is nonzero for $r=k+\ell$.
The coefficient preceding the sum is positive, so
\[
 q_{(\bbeta,k)(\bgamma,\ell)}^{(\bbeta+\bgamma,k+\ell)}\ne0.
\]
Thus, by \cref{prop:Q-criterion}, the $(d+1)$-variate
$Q$-polynomial property holds with respect to $\preceq_Q$.
\end{proof}

\begin{remark}
Remark~10 in \cite{BKZZ2025} conjectures the multivariate $Q$-polynomial property
for generalized Johnson association schemes obtained from finite Gelfand pairs.
The above proof includes this case and further applies to the construction
from an arbitrary commutative association scheme.
\end{remark}

\appendix

\section{The Hahn--multivariate Krawtchouk kernel and orthogonality relations}\label{sec:kernel}

In this appendix, we normalize the second eigenmatrix to give a kernel
corresponding to spherical functions.
The formula in the finite Gelfand pair case is given by
Theorem~6.18 in \cite{CST2006},
and we first check its correspondence with our expression.
We then describe the orthogonality relations for an arbitrary commutative base association scheme.

For the association scheme obtained from a finite Gelfand pair $(G,H)$, put $X=G/H$ and
$y_0=H$. Its primitive idempotents $E_j$ are the orthogonal projections onto the
distinct irreducible summands of $\CC^X$. The \emph{spherical function} associated with $E_j$ is
$\phi_j=(v/m_j)E_je_{y_0}$, normalized to have value $1$ at $y_0$
\cite{BannaiIto1984,CST2006}.
For $(y_0,y)\in R_i$, \eqref{eq:base-PQ} and \eqref{eq:PQ-entry-duality} give
\begin{equation}\label{eq:orbital-spherical-Q}
 \phi_j(y)=\frac{v}{m_j}(E_j)_{y,y_0}
 =\frac{Q_j(i')}{m_j}=\frac{P_i(j)}{k_i}.
\end{equation}
Thus the spherical function associated with $E_j^{\mathsf T}$ has value $Q_j(i)/m_j=\overline{\phi_j(y)}$ at $y$. In the symmetric case, the conjugation can be omitted.

Now let $\frakX$ be an arbitrary commutative association scheme.
Dividing each column of the second eigenmatrix of $\frakJ_{n,h}(\frakX)$
by its multiplicity, define
the \emph{normalized kernel} $\Psi_{\bbeta,k}=Q_{\bbeta,k}/M_{\bbeta,k}$.
Equations~\eqref{eq:second-eigenmatrix}, \eqref{eq:extension-Q},
and \eqref{eq:globalmultiplicity} give
\begin{equation}\label{eq:global-spherical}
 \Psi_{\bbeta,k}(t,\balpha)
 =\frac{Q_{\bbeta,k}(t,\balpha)}{M_{\bbeta,k}}
 =\begin{cases}
 \displaystyle
 \frac{\mathsf H_k(n-|\bbeta|,h-|\bbeta|;t)}{d_{k}^{n-|\bbeta|}}\,
 \frac{Q_{\bbeta}^{[h-t]}(\balpha)}{\mu_{\bbeta}^{(h)}},&|\bbeta|\le h-t,\\[2mm]
 0,&|\bbeta|>h-t.
 \end{cases}
\end{equation}
In particular, if $t=0$ and $\balpha=\bzero$, then
$\mathsf H_k(n-|\bbeta|,h-|\bbeta|;0)=d_{k}^{n-|\bbeta|}$ and
$Q_{\bbeta}^{[h]}(\bzero)=\mu_{\bbeta}^{(h)}$ give
$\Psi_{\bbeta,k}(0,\bzero)=1$.

\begin{remark}
Suppose that $\frakX$ comes from a finite Gelfand pair
$(G,H)$, with base point $y_0\in X=G/H$.
Let $\theta_0$ be the partial map with constant value $y_0$
on a fixed $h$-element subset of $[n]$,
and let $(\theta_0,\theta)\in R_{t,\balpha}$.
By \eqref{eq:orbital-spherical-Q}, the spherical function
associated with $E_{\bbeta,k}$ satisfies
\[
 \phi_{\bbeta,k}(\theta)
 =\frac{|\Theta|}{M_{\bbeta,k}}(E_{\bbeta,k})_{\theta,\theta_0}
 =\frac{Q_{\bbeta,k}(t,\balpha')}{M_{\bbeta,k}}
 =\overline{\Psi_{\bbeta,k}(t,\balpha)},
 \qquad \balpha'=(\alpha_{1'},\ldots,\alpha_{d'}).
\]
Thus $\Psi_{\bbeta,k}(t,\balpha)$ is the value at $\theta$
of the spherical function associated with $E_{\bbeta,k}^{\mathsf T}$.
If $\frakX$ is symmetric, then
$\Psi_{\bbeta,k}(t,\balpha)=\phi_{\bbeta,k}(\theta)$.
This gives the spherical function formula of
\cite[Theorem~6.18]{CST2006} in our notation.
\end{remark}

We now study the normalized kernel \eqref{eq:global-spherical} as a
multivariate orthogonal polynomial on a finite lattice.
We explicitly give the weight and squared norms from the product expression
for the second eigenmatrix, and also verify the orthogonality relations
directly from those of the Johnson scheme and the extension.
Let $n\ge1$ and $0\le h\le n$.
We regard $\Lambda$ and $\Lambda^\ast$ as the index sets
for the variables and degrees, respectively.
We use the normalized Hahn polynomials
\[
 \psi(N,H,k;t)
 :=\frac{\mathsf H_k(N,H;t)}{d_k^N}
 \qquad(0\le t,k\le\min\{H,N-H\}).
\]

\begin{definition}\label[definition]{def:bs-kernel}
By \cref{thm:factorization-main}, the normalized kernel in
\eqref{eq:global-spherical} is expressed as
\begin{equation}\label{eq:bs-kernel}
 \Psi_{\bbeta,k}(t,\balpha)
 =\begin{cases}
 \displaystyle
 \frac{\binom{h-t}{|\bbeta|}}{\binom{h}{|\bbeta|}}\,
 \psi(n-|\bbeta|,h-|\bbeta|,k;t)
 \mathsf F_{d,d}(-\bbeta,-\balpha;-(h-t);\Omega^Q),&|\bbeta|\le h-t,\\[2mm]
 0,&|\bbeta|>h-t
 \end{cases}.
\end{equation}
We call \eqref{eq:bs-kernel} the \emph{Hahn--multivariate Krawtchouk kernel}.
We always have $\Psi_{\bbeta,k}(0,\bzero)=1$.
\end{definition}

In this product, the first two parameters of $\psi$ are $n-|\bbeta|,h-|\bbeta|$
and the length parameter of $\mathsf F_{d,d}$ is $h-t$; thus they depend on the
degree or variable of the other factor.
This way of combining orthogonal polynomials with varying parameters is
reminiscent of Tratnik's \cite{Tratnik1991} product construction.
When $d=1$, $\mathsf F_{d,d}$ is also univariate, and this reduces to the
bivariate polynomials of the nonbinary Johnson association scheme \cite{CVZZ2024}.
For general $d$, the extension factor is a multivariate Krawtchouk polynomial
of Mizukawa--Tanaka type that depends on the entire eigenmatrix of the base association scheme.
Thus the connection with Tratnik type here concerns this product structure
with dependent parameters.
On the region where the kernel is defined to be zero, we do not evaluate
$\mathsf F_{d,d}$ outside its domain.
For nonsymmetric association schemes as well, \eqref{eq:bs-kernel} is defined under the
same convention by using $\Omega^Q$.

\begin{proposition}\label[proposition]{prop:bs-polynomial}
The function $\Psi_{\bbeta,k}$ is the restriction to $\Lambda$ of a polynomial
in $(t,\alpha_1,\ldots,\alpha_d)$ of total degree $|\bbeta|+k$.
\end{proposition}
\begin{proof}
By \eqref{eq:extension-Q-coefficient} and \eqref{eq:extension-Q}, when
$|\bbeta|\le h-t$ we have
\[
 Q_{\bbeta}^{[h-t]}(\balpha)
 =[z_1^{\beta_1}\cdots z_d^{\beta_d}]
 \prod_{i=0}^d\left(1+\sum_{j=1}^dQ_j(i)z_j\right)^{\alpha_i},
 \qquad \alpha_0=h-t-|\balpha|.
\]
Expanding the right-hand side using binomial coefficient polynomials gives
a polynomial in $t,\balpha$ of total degree at most $|\bbeta|$.
When $\alpha_0,\ldots,\alpha_d$ are nonnegative integers with sum
$h-t<|\bbeta|$, the coefficient is zero, so this polynomial also covers the
zero region of \eqref{eq:bs-kernel}.
We can also interpret the expression as
\begin{equation}\label{eq:bs-kernel-coefficient}
 \Psi_{\bbeta,k}(t,\balpha)
 =\psi(n-|\bbeta|,h-|\bbeta|,k;t)\,
 \frac{Q_{\bbeta}^{[h-t]}(\balpha)}{\mu_{\bbeta}^{(h)}}.
\end{equation}
Here $Q_{\bbeta}^{[h-t]}$ for $|\bbeta|>h-t$ denotes the zero value given
by the coefficient above.
Put $b=|\bbeta|$ and
$L_i(\mathbf z)=\sum_{j=1}^d Q_j(i)z_j$.
Using the fact that $\binom{\alpha_i}{r}$ has leading term
$\alpha_i^r/r!$ for $r\ge1$,
we find that the homogeneous component of degree $b$ of
$Q_{\bbeta}^{[h-t]}(\balpha)$, viewed as a polynomial in
$t,\alpha_1,\ldots,\alpha_d$, equals that of
\[
 [\mathbf z^{\bbeta}]
 \frac{1}{b!}
 \left(\sum_{i=0}^d\alpha_iL_i(\mathbf z)\right)^b.
\]
Since $Q_j(0)=m_j$ and $\alpha_0=h-t-|\balpha|$, we have
\[
 \sum_{i=0}^d\alpha_iL_i(\mathbf z)
 =
 \sum_{j=1}^d
 \left(
 hm_j-m_jt+\sum_{i=1}^d(Q_j(i)-m_j)\alpha_i
 \right)z_j.
\]
The constant terms $hm_j$ do not contribute to degree $b$.
Thus the multinomial theorem gives the homogeneous component
of degree $b$ as
\[
 \frac{1}{\beta_1!\cdots\beta_d!}
 \prod_{j=1}^d
 \left(
 -m_jt+\sum_{i=1}^d(Q_j(i)-m_j)\alpha_i
 \right)^{\beta_j}.
\]
Since $m_j\ne0$, each factor is a nonzero linear form.
Since $\psi(n-|\bbeta|,h-|\bbeta|,k;t)$ has degree $k$ in $t$,
the claim follows.
\end{proof}

Put $p_i=k_i/v$ $(0\le i\le d)$.
Since $k_0=1$ and $\sum_ik_i=v$, we have $p_i>0$ and $\sum_ip_i=1$.
We use the following probability weight:
\begin{equation}\label{eq:bs-weight}
 w(t,\balpha)
 =\frac{\binom ht\binom{n-h}{t}}{\binom nh}
 \binom{h-t}{\alpha_0,\alpha_1,\ldots,\alpha_d}
 \prod_{i=0}^dp_i^{\alpha_i},\qquad
 \alpha_0=h-t-|\balpha|.
\end{equation}
This is the product of the distribution of the Johnson distance between two
$h$-element subsets and the multinomial distribution of length $h-t$ with
the distance $t$ fixed.
The multinomial theorem and Vandermonde's identity give $\sum_{\Lambda}w=1$,
and all weights are positive.
Also,
\begin{equation}\label{eq:bs-norm-multiplicity}
 M_{\bbeta,k}
 =\binom{n}{|\bbeta|}\frac{|\bbeta|!}{\beta_1!\cdots\beta_d!}
  d_{k}^{n-|\bbeta|}\prod_{j=1}^dm_j^{\beta_j}>0
\end{equation}
is the multiplicity in \eqref{eq:globalmultiplicity}.

\begin{proposition}\label[proposition]{prop:bs-orthogonality}
For the kernel in \eqref{eq:bs-kernel}, we have
\begin{align}
 \sum_{(t,\balpha)\in \Lambda}w(t,\balpha)
 \Psi_{\bbeta,k}(t,\balpha)
 \overline{\Psi_{\bgamma,\ell}(t,\balpha)}
 &=\frac{\delta_{\bbeta\bgamma}\delta_{k\ell}}{M_{\bbeta,k}},
 \label{eq:bs-orthogonality}\\
 \sum_{(\bbeta,k)\in \Lambda^\ast}M_{\bbeta,k}
 \Psi_{\bbeta,k}(t,\balpha)
 \overline{\Psi_{\bbeta,k}(u,\widetilde{\balpha})}
 &=\frac{\delta_{tu}\delta_{\balpha,\widetilde{\balpha}}}{w(t,\balpha)}.
 \label{eq:bs-dual-orthogonality}
\end{align}
In particular, $\{\Psi_{\bbeta,k}:(\bbeta,k)\in \Lambda^\ast\}$ is an
orthogonal basis of the space of all functions on $\Lambda$,
and the squared norm of $\Psi_{\bbeta,k}$ with respect to $w$ is $1/M_{\bbeta,k}$.
\end{proposition}
\begin{proof}
Applying \eqref{eq:Q-orthogonality} to $\frakX^{[h-t]}$ gives
\[
 v^{-(h-t)}\sum_{\balpha\in D_{h-t}^{\mathrm{ext}}}
 \kappa_{\balpha}^{(h-t)}
 Q_{\bbeta}^{[h-t]}(\balpha)
 \overline{Q_{\bgamma}^{[h-t]}(\balpha)}
 =\mu_{\bbeta}^{(h-t)}\delta_{\bbeta\bgamma}.
\]
Thus, substituting \eqref{eq:bs-kernel-coefficient} and summing over $\balpha$
first gives zero when $\bbeta\ne\bgamma$.
When $\bbeta=\bgamma$, the identities
\[
 \frac{\mu_{\bbeta}^{(h-t)}}{\mu_{\bbeta}^{(h)}}
 =\frac{\binom{h-t}{|\bbeta|}}{\binom{h}{|\bbeta|}},\qquad
 \binom ht\binom{n-h}{t}
 \frac{\binom{h-t}{|\bbeta|}}{\binom{h}{|\bbeta|}}
 =\binom{h-|\bbeta|}{t}\binom{n-h}{t}
\]
show that the remaining sum is
\[
 \frac{1}{\binom nh\mu_{\bbeta}^{(h)}}
 \sum_{t=0}^{\min\{h-|\bbeta|,n-h\}}
 \nu_t^{n-|\bbeta|,h-|\bbeta|}\psi(n-|\bbeta|,h-|\bbeta|,k;t)\psi(n-|\bbeta|,h-|\bbeta|,\ell;t)
 =\frac{\binom{n-|\bbeta|}{h-|\bbeta|}}{\binom nh\mu_{\bbeta}^{(h)}d_{k}^{n-|\bbeta|}}
   \delta_{k\ell}.
\]
Here we used \eqref{eq:E-orthogonality-Krein}.
The final coefficient is $1/M_{\bbeta,k}$, giving \eqref{eq:bs-orthogonality}.
The matrix
\[
 U_{(t,\balpha),(\bbeta,k)}
 =\sqrt{w(t,\balpha)M_{\bbeta,k}}\,
  \Psi_{\bbeta,k}(t,\balpha)
\]
is square and has orthonormal columns.
Hence it is unitary, and the orthogonality of its rows gives
\eqref{eq:bs-dual-orthogonality}.
\end{proof}

\section{Imprimitivity of \texorpdfstring{$\frakJ_{n,h}$}{J(n,h)}}\label{sec:imprimitive}

In this appendix, we describe the partition into blocks given by partial maps
with the same domain and its quotient.
We also relate the polynomial structures induced by the block-type orders
used in the main text to the inheritance theorems of
Higashitani--Kurihara \cite{HigashitaniKurihara2026}.

Let $\mathfrak{X}=(X,\mathcal R,\mathcal I)$ be a commutative association scheme.
For $i,j\in\mathcal I$, put $ij=\{\ell\in\mathcal I:p_{ij}^{\ell}>0\}$,
and, for $C_1,C_2\subseteq\mathcal I$, put
\[
 C_1C_2=\bigcup_{i\in C_1,\ j\in C_2}ij,\qquad
 C_1^{\mathsf T}=\{i':i\in C_1\}.
\]

A nonempty subset $C\subseteq\mathcal I$ is called a \emph{closed subset} if
$C^{\mathsf T}C\subseteq C$.
An association scheme is called \emph{imprimitive} if there exists a closed
subset with $\{0\}\subsetneq C\subsetneq\mathcal I$.

\begin{lemma}[cf.~Lemma~2.3 in \cite{HigashitaniKurihara2026}]\label[lemma]{lem:closed-equivalence}
For a nonempty $C\subseteq\mathcal I$, the following three conditions are equivalent.
\begin{enumerate}[label=\textup{(\roman*)},ref=\textup{(\roman*)}]
\item\label{item:closed-subset}
$C$ is a closed subset.
\item\label{item:closed-point-equivalence}
$R_C=\mathcal R^{-1}(C)=\bigcup_{i\in C}R_i$
is an equivalence relation on $X$.
\item\label{item:closed-label-equivalence}
\[
 i\equiv_Cj
 \quad\Longleftrightarrow\quad
 i'j\cap C\ne\varnothing
\]
is an equivalence relation on $\mathcal I$.
\end{enumerate}
\end{lemma}

Let $C$ be a closed subset.
Denote the equivalence class of $x\in X$ under $R_C$ by $xC$, and the set
of all equivalence classes by $X/C$.
All equivalence classes have the same size
\begin{equation}\label{eq:block-size}
 p=|xC|=\sum_{i\in C}k_i.
\end{equation}
For an equivalence class $x_0C$, define
\[
 \mathcal R_{x_0C}\colon x_0C\times x_0C\longrightarrow C,
 \qquad (x,y)\longmapsto\mathcal R(x,y).
\]
This is a well-defined surjection.
The triple $\left(x_0C,\mathcal R_{x_0C},C\right)$ is called the
\emph{block association scheme} of $\frakX$ on $x_0C$.

Denote the equivalence class of $i\in\mathcal I$ under
\ref{item:closed-label-equivalence} of \cref{lem:closed-equivalence} by $iC$,
and put $\mathcal I/C=\{iC:i\in\mathcal I\}$.
Define
\[
 \mathcal R/C\colon X/C\times X/C\longrightarrow\mathcal I/C,\qquad
 (xC,yC)\longmapsto\mathcal R(x,y)C.
\]
With these relations, $\frakX/C=(X/C,\mathcal R/C,\mathcal I/C)$ is an
association scheme, called the \emph{quotient association scheme} of $\frakX$ by $C$.

Let $j\in\mathcal J$.
By completeness of the primitive idempotents, there exists a unique
$j^{\mathsf T}\in\mathcal J$ such that $E_{j^{\mathsf T}}=E_j^{\mathsf T}$.
Moreover, $(j^{\mathsf T})^{\mathsf T}=j$ and $m_{j^{\mathsf T}}=m_j$.
For $j,\ell\in\mathcal J$, use the Krein parameters to put
$j\circ\ell=\{a\in\mathcal J:q_{j\ell}^{a}>0\}$.
A nonempty $C^*\subseteq\mathcal J$ is called a \emph{closed subset with
respect to the Hadamard product} if $(C^*)^{\mathsf T}\circ C^*\subseteq C^*$.
This is equivalent to $\Span_{\CC}\{E_j:j\in C^*\}$ being a subalgebra under
the Hadamard product \cite{HigashitaniKurihara2026}.

Let $C\subseteq\mathcal I$ be a closed subset, and put $p=|xC|$. By \cite{HigashitaniKurihara2026}, there is a unique subset $C^\ast\subseteq\mathcal J$ such that
\begin{equation}\label{eq:dual-closed-projection}
 M_C=\frac1p\sum_{i\in C}A_i
     =\sum_{j\in C^*}E_j.
\end{equation}
This $C^*$ is called the \emph{dual closed subset} corresponding to $C$.
Moreover, $C^*$ is a closed subset with respect to the Hadamard product.

The set $C^*$ also defines an equivalence relation on $\mathcal J$.
Specifically, for $j,\ell\in\mathcal J$, put
\[
 j\equiv_{C^*}\ell
 \quad\Longleftrightarrow\quad
 (j^{\mathsf T}\circ\ell)\cap C^*\ne\varnothing
 \quad\Longleftrightarrow\quad
 q_{ja}^{\ell}>0\text{ for some }a\in C^*.
\]
Denote the equivalence class of $j$ by $jC^*$ and the set of all equivalence
classes by $\mathcal J/C^*=\{jC^*:j\in\mathcal J\}$.

We apply the above theory to $\frakJ_{n,h}(\mathfrak{X})$.
For $d=0$, that is, when $\frakJ_{n,h}(\mathfrak{X})$ is a Johnson association scheme,
the conditions for imprimitivity are already known, so we consider $d\ge1$ below.

\begin{proposition}\label[proposition]{prop:imprimitive-blocks}
Suppose that $d\ge1$ and $1\le h\le n-1$.
The relation on $\Theta$ given by
\[
 \theta\sim\eta
 \quad\Longleftrightarrow\quad
 \domn\theta=\domn\eta
\]
is a nontrivial equivalence relation.
Moreover,
\[
 C=\{(0,\balpha)\in \Lambda:|\balpha|\le h\}
\]
is a nontrivial closed subset.
Each block has size $v^h$, and there are $\binom nh$ blocks.
The quotient association scheme on the set of blocks is naturally isomorphic to the
Johnson association scheme $J(n,h)$, and the block association scheme on any one block is naturally
isomorphic to the extension of $\frakX$ of length $h$.
\end{proposition}

\begin{proof}
Having equal domains is clearly an equivalence relation.
By \eqref{eq:t-definition}, the conditions $\domn\theta=\domn\eta$ and
$t(\theta,\eta)=0$ are equivalent, and \eqref{eq:alpha-sum} gives
$|\balpha|\le h$ in this case.
This gives the displayed set of relation labels.
By \cref{lem:closed-equivalence}, $C$ is a closed subset.
There are $v^h$ partial maps with domain equal to a fixed $h$-element subset,
and there are $\binom nh$ choices of domain.
If $d\ge1$, the partition of $X\times X$ contains a nonempty relation other
than $R_0$, so $v\ge2$.
Also, if $1\le h\le n-1$, then $\binom nh\ge2$.
Thus both each block and the set of blocks have more than one element,
and the equivalence relation and closed subset are nontrivial.

The relation between two blocks, that is, two $h$-element domains $B_1,B_2$,
is determined by $h-|B_1\cap B_2|$.
Under the identification $(t,\balpha)C\mapsto t$ of the relation labels,
the quotient association scheme is naturally isomorphic to $J(n,h)$.
For a fixed domain $B$, the partial maps in its block are naturally
identified with $X^B$, and the relation between two points is determined
by the numbers of coordinates at which their values form a pair in each $R_i$.
We order $B$ increasingly to identify $X^B$ with $X^h$, and identify
the relation labels by $(0,\balpha)\mapsto\balpha$.
Under these identifications, the block association scheme is naturally
isomorphic to $\frakX^{[h]}$, as defined in \cref{def:extension}.
This is the same extension as in \cite{BKZZ2025}.
\end{proof}

\begin{proposition}\label[proposition]{prop:dual-closed-subset}
Suppose that $d\ge1$ and $1\le h\le n-1$.
Then
\[
 C^*=\{(\bzero,k)\in \Lambda^\ast:0\le k\le\min\{h,n-h\}\}
\]
is the dual closed subset corresponding to the closed subset $C$ of
\cref{prop:imprimitive-blocks}.
More precisely,
\begin{equation}\label{eq:GJ-dual-projection}
 \frac1{v^h}
 \sum_{\substack{\balpha\in\NN^d\\|\balpha|\le h}}A_{0,\balpha}
 =
 \sum_{k=0}^{\min\{h,n-h\}}E_{\bzero,k}.
\end{equation}
Consequently, the linear span of the primitive idempotents indexed by $C^*$
is closed under the Hadamard product.
\end{proposition}

\begin{proof}
If $\bbeta=\bzero$, then $|\bbeta|=0$, $\mu_{\bzero}^{(h-t)}=1$, and
$\mathsf F_{d,d}(\bzero,-\balpha;-(h-t);\Omega^Q)=1$.
Thus \eqref{eq:projector-kernel} becomes
\[
 (E_{\bzero,k})_{\theta,\eta}
 =v^{-h}
 \bigl(E_k^{J(n,h)}\bigr)_{\domn\theta,\domn\eta}.
\]
Since the sum of the primitive idempotents of the Johnson association scheme is the identity matrix, we have
\[
 \sum_{k=0}^{\min\{h,n-h\}}E_{\bzero,k}
 =
 \frac1{v^h}
 \sum_{\substack{\balpha\in\NN^d\\|\balpha|\le h}}
 A_{0,\balpha}.
\]
Indeed, the $(\theta,\eta)$ entry on each side is $v^{-h}$ if
$\domn\theta=\domn\eta$, and $0$ otherwise.
The block size of the closed subset in \cref{prop:imprimitive-blocks} is $v^h$,
so the right-hand side is $M_C$ in \eqref{eq:dual-closed-projection}.
Hence $C^*$ is the dual closed subset
corresponding to this closed subset, and the linear span of its primitive
idempotents is closed under the Hadamard product.
\end{proof}

The order \eqref{eq:P-block-order} is of $1$-block type, and
\eqref{eq:Q-block-order} is of $d$-block type.
Thus $C$ and $C^*$ in \cref{prop:imprimitive-blocks,prop:dual-closed-subset}
correspond to the inheritance results in Theorems~4.6 and~4.7 of
Higashitani--Kurihara \cite{HigashitaniKurihara2026}.
We give their content explicitly in terms of our generators and normalizations.
The following proposition follows by applying these theorems to $\frakJ_{n,h}(\frakX)$,
so we omit its proof.

\begin{proposition}\label[proposition]{prop:block-quotient-polynomial}
Let $d\ge1$ and $1\le h\le n-1$, and fix $B\in\binom{[n]}h$.
Let $\Theta_B=\{\theta\in\Theta:\domn\theta=B\}$ be the corresponding block.
Write $A_{\balpha}^{[h]}$ and $E_{\bbeta}^{[h]}$ for the adjacency matrices
and primitive idempotents of $\Theta_B$, and $D_t$ and $E_k^J$ for those of
the quotient $J(n,h)$.
Then the following hold.
\begin{enumerate}[label=\textup{(\roman*)},ref=\textup{(\roman*)}]
\item
The block association scheme $\frakX^{[h]}$ is $d$-variate $P$- and $Q$-polynomial on
$D_h^{\mathrm{ext}}=\{\balpha\in\NN^d:|\balpha|\le h\}$ with respect to
$\preceq_{\grlex}$.
The respective generators are $A_{\be_i}^{[h]}$ $(1\le i\le d)$ and
$E_{\be_j}^{[h]}$ $(1\le j\le d)$.
The quotient $J(n,h)$ is univariate $P$- and $Q$-polynomial with respect to
the usual order on $0\le t,k\le\min\{h,n-h\}$, with generators $D_1$ and $E_1^J$.
Their eigenmatrices are
\begin{align}
 P_{\balpha}^{\mathrm{bl}}(\bbeta)&=P_{\balpha}^{[h]}(\bbeta),&
 Q_{\bbeta}^{\mathrm{bl}}(\balpha)&=Q_{\bbeta}^{[h]}(\balpha),
 \label{eq:block-eigenmatrices}\\
 P_t^{\mathrm{qt}}(k)&=\mathsf E_t(n,h;k),&
 Q_k^{\mathrm{qt}}(t)&=\mathsf H_k(n,h;t),
 \label{eq:quotient-eigenmatrices}
\end{align}
where the right-hand sides for the block are given by the multivariate
Krawtchouk polynomials in \eqref{eq:extension-P} and \eqref{eq:extension-Q}.
\item
Let $f_{t,\balpha}(x_0,x_1,\ldots,x_d)$ and
$g_{\bbeta,k}(y_1,\ldots,y_d,z)$ be the associated polynomials of
$\frakJ_{n,h}(\frakX)$; that is,
\[
\begin{aligned}
 A_{t,\balpha}&=f_{t,\balpha}(A_{1,\bzero},A_{0,\be_1},\ldots,A_{0,\be_d}),\\
 F_{\bbeta,k}&=g_{\bbeta,k}(F_{\be_1,0},\ldots,F_{\be_d,0},F_{\bzero,1}),
\end{aligned}
\]
where $F_{\bbeta,k}=|\Theta|E_{\bbeta,k}$ and the latter is evaluated using
the Hadamard product.
The associated polynomials of the block and the quotient can be taken, respectively, as
\begin{align}
 f_{\balpha}^{\mathrm{bl}}(\boldsymbol{x})
 &=f_{0,\balpha}(0,\boldsymbol{x}),
 \label{eq:block-P-polynomials}\\
 f_t^{\mathrm{qt}}(x)
 &=\frac1{v^h}\sum_{\substack{\balpha\in\NN^d\\|\balpha|\le h-t}}
    f_{t,\balpha}(vx,hk_1,\ldots,hk_d),
 \label{eq:quotient-P-polynomials}\\
 g_{\bbeta}^{\mathrm{bl}}(\boldsymbol{y})
 &=\frac1{\binom nh}
   \sum_{k=0}^{\min\{h-|\bbeta|,n-h\}}
   g_{\bbeta,k}\left(\frac nh y_1,\ldots,\frac nh y_d,n-1\right),
 \label{eq:block-Q-polynomials}\\
 g_k^{\mathrm{qt}}(z)&=g_{\bzero,k}(0,\ldots,0,z).
 \label{eq:quotient-Q-polynomials}
\end{align}
On the $Q$ side, the variables are replaced by $v^hE_{\be_j}^{[h]}$ for the
block and by $\binom nh E_1^J$ for the quotient.
\end{enumerate}
\end{proposition}

\section*{Acknowledgments}
The author is supported by JSPS KAKENHI Grant Number JP24K06830.

\bibliographystyle{gjplain}
\bibliography{ref}

\end{document}